\documentclass[12 pt,oneside,reqno]{amsart}

\usepackage[a4paper, total={6in, 8in}]{geometry}
\usepackage[dvips]{graphicx}
\usepackage{calc}
\usepackage{color}
\usepackage{amsmath}
\usepackage{amssymb}
\usepackage{amscd}
\usepackage{amsthm}
\usepackage{esdiff}
\usepackage{amsbsy}
\usepackage{delarray}
\usepackage{enumerate}
\usepackage[T1]{fontenc}
\usepackage{inputenc}
\usepackage{enumerate}
\usepackage{hyperref}

\begin{document}



\setlength{\parindent}{5mm}
\renewcommand{\leq}{\leqslant}
\renewcommand{\geq}{\geqslant}
\newcommand{\N}{\mathbb{N}}
\newcommand{\sph}{\mathbb{S}}
\newcommand{\Z}{\mathbb{Z}}
\newcommand{\R}{\mathbb{R}}
\newcommand{\C}{\mathbb{C}}
\newcommand{\F}{\mathbb{F}}
\newcommand{\g}{\mathfrak{g}}
\newcommand{\h}{\mathfrak{h}}
\newcommand{\K}{\mathbb{K}}
\newcommand{\RN}{\mathbb{R}^{2n}}
\newcommand{\ci}{c^{\infty}}
\newcommand{\derive}[2]{\frac{\partial{#1}}{\partial{#2}}}
\renewcommand{\S}{\mathcal{S}}
\renewcommand{\H}{\mathbb{H}}
\newcommand{\eps}{\varepsilon}
\newcommand{\chzrel}{c_{\mathrm{LR}}}
\newcommand{\ham}{Hamiltonian}
\newcommand{\sa}{symplectically aspherical}
\newcommand{\ovl}{\overline}
\newcommand{\wt}{\widetilde}
\newcommand{\hamd}{Hamiltonian diffeomorphism}

\theoremstyle{plain}
\newtheorem{theo}{Theorem}
\newtheorem{prop}[theo]{Proposition}
\newtheorem{lemma}[theo]{Lemma}
\newtheorem{claim}[theo]{Claim}
\newtheorem{definition}[theo]{Definition}
\newtheorem*{notation*}{Notation}
\newtheorem*{notations*}{Notations}
\newtheorem{corol}[theo]{Corollary}
\newtheorem{conj}[theo]{Conjecture}

\newenvironment{demo}[1][]{\addvspace{8mm} \emph{Proof #1.
    ~~}}{~~~$\Box$\bigskip}

\newlength{\espaceavantspecialthm}
\newlength{\espaceapresspecialthm}
\setlength{\espaceavantspecialthm}{\topsep} \setlength{\espaceapresspecialthm}{\topsep}

\newenvironment{example}[1][]{
\vskip \espaceavantspecialthm \noindent \textsc{Example
#1.} }%
{\vskip \espaceapresspecialthm}

\newenvironment{question}[1][]{
\vskip \espaceavantspecialthm \noindent \textsc{Question
#1.} }%
{\vskip \espaceapresspecialthm}

\newenvironment{remark}[1][]{\refstepcounter{theo} 
\vskip \espaceavantspecialthm \noindent \textsc{Remark~\thetheo
#1.} }%
{\vskip \espaceapresspecialthm}

\def\bb#1{\mathbb{#1}} \def\m#1{\mathcal{#1}}

\def\momeg{(M,\omega)}
\def\co{\colon\thinspace}
\def\Homeo{\mathrm{Homeo}}
\def\Hameo{\mathrm{Hameo}}
\def\Diffeo{\mathrm{Diffeo}}
\def\Symp{\mathrm{Symp}}
\def\Sympeo{\mathrm{Sympeo}}
\def\id{\mathrm{id}}
\def\Im{\mathrm{Im}}
\newcommand{\norm}[1]{||#1||}
\def\Ham{\mathrm{Ham}}
\def\lagham#1{\mathcal{L}^\mathrm{Ham}({#1})}
\def\Hamtilde{\widetilde{\mathrm{Ham}}}
\def\cOlag#1{\mathrm{Sympeo}({#1})}
\def\Crit{\mathrm{Crit}}
\def\dim{\mathrm{dim}}
\def\Spec{\mathrm{Spec}}
\def\osc{\mathrm{osc}}
\def\Cal{\mathrm{Cal}}
\def\Fix{\mathrm{Fix}}
\def\det{\mathrm{det}}
\def\Ker{\mathrm{Ker}}
\def\Supp{\mathrm{Supp}}

\title[]{On $C^0$-continuity of the spectral norm for symplectically non-aspherical manifolds}
\author{Yusuke Kawamoto}

\address{Yusuke Kawamoto: D\'epartement de Math\'ematiques et Applications, \'Ecole Normale Sup\'erieure, 45 rue d'Ulm, F 75230 Paris cedex 05}
\email{yusukekawamoto81@gmail.com, yusuke.kawamoto@ens.fr}

\maketitle

\begin{abstract}
The purpose of this paper is to study the relation between the $C^0$-topology and the topology induced by the spectral norm on the group of Hamiltonian diffeomorphisms of a closed symplectic manifold. Following the approach of Buhovsky-Humili\`ere-Seyfaddini, we prove the $C^0$-continuity of the spectral norm for complex projective spaces and negative monotone symplectic manifolds. The case of complex projective spaces provides an alternative approach to the $C^0$-continuity of the spectral norm proven by Shelukhin. We also prove a partial $C^0$-continuity of the spectral norm for rational symplectic manifolds. Some applications such as the Arnold conjecture in the context of $C^0$-symplectic topology are also discussed.
 \end{abstract}

\tableofcontents

\section{Introduction}\label{sec:introduction}
The study of topological properties of the group of Hamiltonian diffeomorphisms of a symplectic manifold has been one of the central topics in symplectic topology. The group of Hamiltonian diffeomorphisms is known to carry different metrics such as the Hofer metric, the spectral metric and the $C^0$-metric and their relations have been studied extensively. This paper studies the relation between the $C^0$-topology and the topology induced by the spectral metric. More precisely, we study the $C^0$-continuity of the spectral norm which has been already verified for certain cases: for $\mathbb{R}^{2n}$ by Viterbo \cite{[Vit92]}, for closed surfaces by Seyfaddini \cite{[Sey13-1]}, for symplectically aspherical manifolds by Buhovsky-Humili\`ere-Seyfaddini \cite{[BHS18b]} and for complex projective spaces by Shelukhin \cite{[Sh18]}.  In this paper, we push the method developed by Buhovsky-Humili\`ere-Seyfaddini \cite{[BHS18b]} forward to the symplectically non-aspherical setting and confirm the $C^0$-continuity of the spectral norm for negative monotone symplectic manifolds. We also obtain a partial $C^0$-continuity of the spectral norm for rational symplectic manifolds and an alternative proof of the $C^0$-continuity of the spectral norm for complex projective spaces.

\subsection{Set-up}
Throughout this paper, $(M,\omega)$ will denote a closed symplectic manifold. A symplectic manifold $(M,\omega)$ is called
\begin{itemize}
\item rational if $\langle \omega,\pi_2(M)  \rangle =\lambda_0 \mathbb{Z}$ for some constant $\lambda_0>0$. We refer to the constant $\lambda_0$ as the rationality constant.
\item monotone (resp. negative monotone) if $\omega|_{\pi_2(M)}=\lambda \cdot c_1|_{\pi_2(M)}$ for some positive (resp. negative) constant $\lambda$ where $c_1:=c_1(TM)$ denotes the first Chern class of $TM$. We refer to the constant $\lambda$ as the monotonicity constant. 
\item symplectically aspherical when $\omega|_{\pi_2(M)}= c_1|_{\pi_2(M)}=0$.
\end{itemize}

The positive generator of $\langle c_1,\pi_2(M)  \rangle \subset \mathbb{Z}$ is called the minimal Chern number $N_M$ i.e. 
$$\langle c_1,\pi_2(M)  \rangle=N_M \mathbb{Z},\ N_M>0.$$

\begin{example}
\begin{itemize}
\item The complex projective space equipped with the standard Fubini-Study form $(\mathbb{C}P^n,\omega_{FS})$ is monotone and its minimal Chern number $N_{\mathbb{C}P^n}$ is $n+1$.
\item The degree $k$ Fermat hypersurfaces of $\mathbb{C}P^{n+1}$
$$M:=\{(z_0:z_1:\cdots:z_n) \in \mathbb{C}P^{n+1}:z_0 ^k+z_1 ^k+z_2 ^k+\cdots+z_n ^k=0\}$$
is negative monotone for $k>n+1$. The minimal Chern number $N_M$ is $|k-(n+2)|$ if $k \neq n+2$ and $+\infty$ otherwise.
\end{itemize}
\end{example}

A Hamiltonian $H$ on $M$ is a smooth time dependent function $H: \mathbb{R}/\mathbb{Z} \times M\to \mathbb{R}$. We define its Hamiltonian vector field $X_{H_t}$ by $-dH_t=\omega(X_{H_t},\cdot).$ The Hamiltonian flow of $H$, denoted by $(\phi_H ^t)_{t \in \R}$, is by definition the flow of $X_H$. Its time-one map $\phi_H ^{1}$ is called the Hamiltonian diffeomorphism of $H$ and will be denoted by $\phi_H$. The set of Hamiltonian diffeomorphisms and its universal cover will be denoted respectively by $\Ham(M,\omega)$ and $\widetilde{\Ham}(M,\omega)$.

\subsection{$C^0$-topology}
We define the $C^0$-metric by 
$$\phi, \psi \in \Ham(M,\omega),\ d_{C^0}(\phi,\psi):=\max_{x\in M}d(\phi(x),\psi(x))$$
where $d$ denotes the distance on $M$ induced by the Riemannian metric on $M$. Note that the topology induced by the $C^0$-distance is independent of the choice of a Riemannian metric. We denote the $C^0$-closure of $\Ham(M,\omega)$ in the group of homeomorphisms of $M$ by $\overline{\Ham}(M,\omega)$. Their elements are called the Hamiltonian homeomorphisms. Hamiltonian homeomorphisms are central objects in $C^0$-symplectic topology.

\subsection{Spectral norms}
We roughly outline the notion of the spectral norm. For precise definitions, we refer to Section \ref{spec inv}. First of all, a Hamiltonian $H\in C^{\infty}(\mathbb{R}/\mathbb{Z}\times M,\mathbb{R})$ is called non-degenerate if for each $x \in \Fix (\phi_H)$, the set of eigenvalues of $d \phi_H (x):T_x M \to T_x M$ does not include $1$. For a non-degenerate Hamiltonian $H$ and a fixed ground field $\K$ (see Remark \ref{ground field} for more information about the choice of a ground field), one can define the Floer homology group $HF(H)=HF(H;\K)$ as well as their filtration with respect to the action functional which will be denoted by $\{ HF^\tau (H) \}_{\tau \in \R}$. For each $\tau \in \R$, we denote the natural map induced by the inclusion map of the chain complex by $i_\ast ^\tau$:
$$i_\ast ^\tau: HF^\tau (H) \to HF (H).$$
The quantum homology ring is defined by
$$QH_\ast (M;\K) := H_\ast(M;\K) \otimes_\K \Lambda $$ 
where
$$\Gamma:=\pi_2(M)/\Ker(\omega) \cap \Ker(c_1),$$
$$\Lambda:= \{\sum_{A \in \Gamma} a_A \otimes e^A: a_A \in \K, (\forall \tau \in \R, \#\{a_A \neq0:\omega(A) <\tau \}<+\infty) \}.$$
The ring structure of $QH_\ast(M;\K)$  is given by the quantum product $\ast$: for its definition, see Section \ref{QH}). Floer homology group is ring isomorphic to the quantum homology ring $QH_\ast (M;\K)$ by the PSS-isomorphism
$$\Phi_{PSS,H;\K}: QH_\ast (M;\K) \to HF(H).$$
For a Hamiltonian $H$ and $a \in  QH_\ast (M;\K) \backslash \{0\} $, the spectral invariant of $H$ and $a$ is defined by
$$c(H,a):= \inf \{\tau : \Phi_{PSS,H;\K} (a) \in \Im (i_\ast ^\tau) \}.$$
The spectral norm of a Hamiltonian $H$ is defined by
$$\gamma(H):=c(H,[M])+c(\overline{H},[M])$$
where $\overline{H}(t,x):=-H(t,\phi_H ^t(x))$ which is a Hamiltonian that generates the Hamiltonian flow 
$$t \mapsto (\phi_H ^t) ^{-1}.$$

Since $\gamma$ is invariant under homotopy i.e. if $\phi_H ^t \sim \phi_G ^t$ rel. endpoints, then $\gamma(H)=\gamma(G)$, it can be seen as a map defined on the universal cover of $\Ham(M,\omega)$, namely 
$$\gamma:\widetilde{\Ham}(M,\omega)\to \mathbb{R}.$$ 

We define spectral norms for Hamiltonian diffeomorphisms by
$$\gamma:\Ham(M,\omega)\to \mathbb{R},$$
$$\gamma(\phi):=\inf_{\phi=\phi_H} \gamma(H).$$

\subsection{Main results}
Throughout the paper $\lambda_0>0$ denotes the rationality constant i.e. $\langle \omega, \pi_2(M) \rangle =\lambda_0 \mathbb{Z}$. We first state our result for rational symplectic manifolds.

\begin{theo}\label{rational}
Let $(M,\omega)$ be a rational symplectic manifold. For any $\varepsilon>0,$ there exists $\delta>0$ such that if $d_{C^0}(\id,\phi_H)<\delta,$ then 
$$|\gamma(H)-l\cdot \lambda_0|<\varepsilon$$
for some integer $l\in \mathbb{Z}$ depending on $H$.
\end{theo}

Theorem \ref{rational} gives us the candidates of the value of the spectral norm of a $C^0$-small $\phi_H$. When the values of spectral norms are bounded by a number strictly smaller than $\lambda_0$, then Theorem \ref{rational} implies the $C^0$-continuity. The complex projective space $\C P^n$ meets this condition.

\begin{theo}\label{C0-conti of gamma:CPn}
Let $(\mathbb{C}P^n,\omega_{FS})$ be the complex projective space equipped with the Fubini-Study form.
\begin{enumerate}
\item For any $\phi \in \Ham(\mathbb{C}P^n,\omega_{FS}),$ 
$$\gamma (\phi)\leq \frac{n}{n+1}\cdot \lambda_0$$
where $\lambda_0$ denotes the rationality constant.
\item The spectral norm is $C^0$-continuous i.e.
$$\gamma :(\Ham(\mathbb{C}P^n,\omega_{FS}),d_{C^0}) \to \mathbb{R}$$
is continuous. Moreover, $\gamma$ extends continuously to $\overline{\Ham}(\mathbb{C}P^n,\omega_{FS}).$
\end{enumerate}
\end{theo}

\begin{remark}
\begin{enumerate}
\item Theorem \ref{C0-conti of gamma:CPn} is already proven in other papers: (1) appears as Theorem G (2) in \cite{[KS18]} and (2) appears as Theorem C in \cite{[Sh18]}. Shelukhin's argument in \cite{[Sh18]}, which is different from ours, is based on barcode techniques and is specific for $\C P^n$. 

\item We will prove an a priori more general result in Section \ref{monotone proof}. 
\end{enumerate}
\end{remark}

For general monotone symplectic manifolds, instead of the $C^0$-continuity, we only obtain the following $C^0$-control of the spectral norm.

\begin{theo}\label{monotone bound extra}
Let $(M,\omega)$ be a monotone symplectic manifold. 
\begin{enumerate}
\item For any $\varepsilon>0,$ there exists $\delta>0$ such that if $d_{C^0}(\id,\phi_H)<\delta,$ then
$$\gamma(H)< \frac{\dim(M)}{N_M}\cdot \lambda_0+\varepsilon.$$
\item If $N_M> \dim(M)$, then the spectral norm is $C^0$-continuous i.e.
$$\gamma:(\Ham(M,\omega),d_{C^0}) \to \mathbb{R}$$
is continuous. Moreover, $\gamma$ extends continuously to $\overline{\Ham}(M,\omega).$
\end{enumerate}
\end{theo}

\begin{remark}
\begin{enumerate}
\item The author does not know any example satisfying the assumptions in Theorem \ref{monotone bound extra} (2). Note that Theorem \ref{monotone bound extra} (2) follows immediately from Theorem \ref{rational} and Theorem \ref{monotone bound extra} (1).

\item Theorem \ref{monotone bound extra} applies to spectral norms of any ground ring $R$ (i.e. a commutative ring with unit). See Remark \ref{ground field} for a comment on the choice of ground rings/fields.
\end{enumerate}
\end{remark}

We now consider the case of negative monotone symplectic manifolds.

\begin{theo}\label{neg mono}
Let $(M,\omega)$ be a negative monotone symplectic manifold.
\begin{enumerate}
\item For any $\varepsilon>0$, there exists $\delta>0$ such that if $d_{C^0}(\id,\phi_H)<\delta$, then 
$$\gamma(H)<\varepsilon.$$
In particular, if $\phi_H=\phi_G$ for $H,G\in C^{\infty}(\mathbb{R}/\mathbb{Z}\times M,\mathbb{R})$, then 
$$\gamma(H)=\gamma(G)$$

i.e. $\gamma:\widetilde{\Ham}(M,\omega)\to \mathbb{R}$
descends to 
$$\gamma:\Ham(M,\omega) \to \mathbb{R}.$$

\item The spectral norm is $C^0$-continuous i.e. 
$$\gamma:(\Ham(M,\omega),d_{C^0}) \to \mathbb{R}$$
is continuous. Moreover, $\gamma$ extends continuously to $\overline{\Ham}(M,\omega).$
\end{enumerate}
\end{theo}

\begin{remark}
The independence of the spectral norm of the choice of Hamiltonian follows also from Lemma 3.2.(iv) in \cite{[McD10]}.
\end{remark}

\subsection{Application 1: $C^0$-continuity of barcodes}
Barcodes are roughly speaking finite sets of intervals which are bounded from below but can be unbounded from above. The set of barcodes carries a metric called the bottleneck distance denoted by $d_{bot}$. Barcodes have been a common tool in topological data analysis. Polterovich-Shelukhin brought barcodes into symplectic topology in \cite{[PS16]} where they defined barcodes of (non-degenerate) Hamiltonian diffeomorphisms on symplectically aspherical manifolds and found applications to Hofer geometry. Later, as we will explain in Section \ref{barcodes}, the definition of barcodes was extended to Hamiltonian diffeomorphisms on (negative) monotone symplectic manifolds \cite{[LSV18]}, \cite{[PSS17]} after considering a completion of the set of barcodes with respect to the bottleneck distance which we will denote by $\widehat{Barcodes}$. An estimate of the bottleneck distance due to Kislev--Shelukhin \cite{[KS18]} (see the inequality \ref{KS ineq} in Section \ref{barcodes}) combined with the $C^0$-continuity of the spectral norm implies the $C^0$-continuity of barcodes for negative monotone symplectic manifolds.

\begin{corol}\label{neg mono barcode}
Let $(M,\omega)$ be a negative monotone symplectic manifold. The barcode map is $C^0$-continuous i.e.
$$B:(\Ham(M,\omega),d_{C^0})\to (\widehat{Barcodes},d_{bot})$$
is continuous. Moreover, $B$ extends continuously to $\overline{\Ham}(M,\omega).$
\end{corol}

\begin{remark}
Of course, Theorem \ref{C0-conti of gamma:CPn} (2) directly implies the $C^0$-continuity of barcodes in the case of $(\mathbb{C}P^n,\omega_{FS})$. This is proven by Shelukhin in Corollary 6 in \cite{[Sh18]}.
\end{remark}

\subsection{Application 2: The $C^0$-Arnold conjecture}\label{App2}

The Arnold conjecture has been historically one of the central topics in symplectic geometry.

\begin{conj}$($The Arnold conjecture$)$

Let $(M^{2n},\omega)$ be a closed symplectic manifold. 
\begin{enumerate}
\item For a non-degenerate $\phi \in \Ham(M,\omega),$
$$\# \Fix (\phi)\geq \sum_{j } \dim_{\mathbb{C}} H_j(M;\mathbb{C}).$$

\item For $\phi \in \Ham(M,\omega),$
$$\# \Fix(\phi)\geq cl(M)$$
where
$$cl(M):=\#\max \{k+1:\exists a_1,a_2,\cdots,a_k\in H_{\ast<2n}(M)\ s.t.\ a_1\cap a_2\cap \cdots \cap a_k\neq 0\}$$
and $\cap$ denotes the intersection product.
\end{enumerate}
\end{conj}

Since the advent of Floer homology, there has been a huge progress in the two versions of the Arnold conjecture: (1) is now completely settled \cite{[FO99]}, \cite{[LT98]} and (2) has been confirmed for symplectically aspherical manifolds \cite{[Fl89]}, $\mathbb{C}P^n$ \cite{[For85]} \cite{[ForW85]} and negative monotone symplectic manifolds with $N_M\geq n$ \cite{[LO94]}.

It caught attention whether or not the Arnold conjecture is $C^0$-robust i.e. if Hamiltonian homeomorphisms satisfy similar properties. For closed surfaces, this question was answered in the positive by Matsumoto \cite{[Ma00]}. However, Buhovsky-Humili\`ere-Seyfaddini \cite{[BHS18a]} discovered that in higher dimension, this turns out not to be the case.

\begin{theo}$($\cite{[BHS18a]}$)$

Let $(M,\omega)$ be any closed symplectic manifold of dimension $\geq 4$. There exists a Hamiltonian homeomorphism $\phi \in \overline{\Ham}(M,\omega)$ such that
$$\# \Fix(\phi)=1.$$
\end{theo}

In their subsequent paper \cite{[BHS18b]}, Buhovsky-Humili\`ere-Seyfaddini have reformulated the Arnold conjecture in a way that is more suited to study the rigidity of Hamiltonian homeomorphisms when the ambient manifold is {\sa}. We will follow their idea to obtain similar results for symplectic manifolds that are not {\sa} by using the quantum product $\ast$ of $QH_\ast(M;\K)$ (for its definition, see Section \ref{QH}).

\begin{definition}
Let $(M^{2n},\omega)$ be a symplectic manifold. Let $a,b\in H_\ast(M;\K)\backslash \{0\}$. For a Hamiltonian $H$, define 
$$\sigma_{a, a \ast b}(H):=c(H,a)-c(H, a \ast b)$$
and for a {\hamd} $\phi$, define
$$\sigma_{a, a \ast b}(\phi):=\inf_{\phi_H=\phi} \sigma_{a, a\ast b}(H).$$

\end{definition}

For $(M,\omega)$ for which $\gamma$ is $C^0$-continuous (e.g. negative monotone symplectic manifolds and $(\mathbb{C}P^n,\omega_{FS})$), $\sigma_{a,a \ast b}$ turns out to be $C^0$-continuous and it extends continuously to $\overline{\Ham}(M,\omega)$: see Section \ref{section-arnold} for details.

\begin{theo}\label{theo-arnold}
Let $(M^{2n},\omega)$ be either a negative monotone symplectic manifold or $(\mathbb{C}P^n,\omega_{FS})$. For $\phi \in \overline{\Ham}(M,\omega)$, if there exist homology classes $a,b\in H_{\ast}(M;\K) \backslash \{0\}, b \neq [M]$ such that 
$$\sigma_{a,a\ast b}(\phi)=0,$$
then $\Fix(\phi)$ is homologically non-trivial, hence it is an infinite set.
\end{theo}

\begin{remark}\label{ast cap}
\begin{enumerate}
\item Recall that, a subset $A\subset M$ is homologically non-trivial if for every open neighborhood $U$ of $A$ the map $i_\ast: H_j(U;\K) \to H_j(M;\K)$, induced by the inclusion $i:U\to M$, is non-trivial for some $j>0$. Homologically non-trivial sets are infinite sets.

\item In \cite{[How12]}, Howard considers the smooth version of this statement.

\end{enumerate}
\end{remark}

\subsection{Application 3: The displaced disks problem}

A topological group $G$ is a Rokhlin group if it possesses a dense conjugacy class i.e. for some $ \phi \in G,$ $Conj(\phi):=\{\psi^{-1}\phi\psi:\psi \in G\}$ is dense. B\'eguin-Crovisier-Le Roux formulated the following question so-called the "displaced disks problem", in order to study whether or not the group of area-preserving homeomorphisms on a sphere is a Rokhlin group.

\begin{question}$($B\'eguin-Crovisier-Le Roux$)$

Let 
$$\mathcal{G}_r:=\{\phi \in \overline{\Ham} (M,\omega): \phi(f(B_r))\cap f(B_r)=\emptyset \}$$
for any $r>0$ where $f:B_r\to (M,\omega)$ is a symplectic embedding. Does the $C^0$-closure of $\mathcal{G}_r$ contain the identity element for some $r>0$?
\end{question}

This original question which was for $(M,\omega)=(S^2,\omega_{area})$ was solved by Seyfaddini in \cite{[Sey13-2]} as a consequence of his earlier result on $C^0$-continuity of spectral norms for closed surfaces \cite{[Sey13-1]}. Other cases, also deduced by $C^0$-continuity of spectral norms, has also been considered: \cite{[BHS18b]} deals with symplectically aspherical manifolds and \cite{[Sh18]} treats $\mathbb{C}P^n$. Here we add the case of negative monotone symplectic manifolds.

\begin{theo}\label{ddp}
Let $(M,\omega)$ be a negative monotone symplectic manifold. For any $r>0$, there exists $\delta>0$ such that if $\phi \in  \overline{\Ham}(M,\omega)$ displaces a symplectically embedded ball of radius $r$, then $d_{C^0}(\id,\phi)>\delta.$
\end{theo}

We obtain the following as a direct consequence. 

\begin{corol}\label{rokhlin}
Let $(M,\omega)$ be a negative monotone symplectic manifold. The group $\overline{\Ham}(M,\omega)$ seen as a topological group with respect to the $C^0$-topology is not a Rokhlin group. 
\end{corol}

\begin{remark}
The case of $(\mathbb{C}P^n,\omega_{FS})$ was considered by Shelukhin in \cite{[Sh18]}.
\end{remark}

\subsection{Acknowledgments}
I would like to thank Sobhan Seyfaddini for having generously shared some ideas and for continuous encouragements. I also would like to express my gratitude to Claude Viterbo for numerous useful discussions and to R\'emi Leclercq for his stimulating comments on the earlier version of the paper. I am also indebted to Egor Shelukhin for his very helpful suggestions and interesting comments on the earlier version of this paper. The referee has contributed a lot to the improvement of the exposition. I owe a lot to her/him. Most part of this work was carried out at the MSRI, Berkeley while I was in residence during the Fall 2018 semester. I greatly benefited from the lively research atmosphere of the MSRI and would like to thank the institute for providing an excellent research environment.

\section{Preliminaries}
In this section, we briefly review the notions and basic propositions needed later in the proof. For further details, we refer to \cite{[MS04]}.

Let $(M,\omega)$ be a symplectic manifold. A Hamiltonian $H$ on $M$ is a smooth time dependent function $H:\mathbb{R}/\mathbb{Z}\times M\to \mathbb{R}$. We define its Hamiltonian vector field $X_{H_t}$ by $-dH_t=\omega(X_{H_t},\cdot).$ The Hamiltonian flow of $H$, denoted by $\phi_H ^t$, is by definition the flow of $X_H$. A Hamiltonian diffeomorphism is a diffeomorphism which arises as the time-one map of a Hamiltonian flow. The set of all Hamiltonian diffeomorphisms is denoted by $\Ham(M,\omega)$.

 Denote the set of smooth contractible loops in $M$ by $\mathcal{L}M$ and consider its universal cover. Two elements in the universal cover, say $[z_1,w_1]$ and $[z_2,w_2]$, are called equivalent if their boundary sum $w_1\# \overline{w_2}$ i.e. the sphere obtained by gluing $w_1$ and $w_2$ along their common boundary with the orientation on $w_2$ reversed, satisfies 
 $$\omega (w_1\# \overline{w_2})=0,\ c_1(w_1\# \overline{w_2})=0.$$
 We denote by $\widetilde{\mathcal{L}_0M}$ the space of equivalence classes. For a Hamiltonian $H$, define the action functional $\mathcal{A}_H:\widetilde{\mathcal{L}M}\to \mathbb{R}$ by
$$\mathcal{A}_H([z,w]):=\int_0 ^1 H(t,z(t))dt-\int_{D^2} w^{\ast}\omega$$
where $w:D^2\to M$ is a capping of $z:\mathbb{R}/\mathbb{Z}\to M$. Note that in general, the action functional depends on the capping and not only on the loop. Critical points of this functional are precisely the capped 1-periodic Hamiltonian orbits of $H$ which will be denoted by $\widetilde{\mathcal{P}}(H)$. The set of critical values of $\mathcal{A}_H$ is called the action spectrum and is denoted by $\Spec(H)$:
$$\Spec(H):=\{\mathcal{A}_H(\widetilde{z}):\widetilde{z} \in \widetilde{\mathcal{P}}(H)\}.$$

We briefly explain some notions of indices used later to construct Floer homology. The Maslov index 
$$\mu: \pi_1(Sp(2n)) \to \mathbb{Z}$$
maps a loop of symplectic matrices to an integer. Given a loop of Hamiltonian diffeomorphims $(\psi^t)_{t \in [0,1]} \in \pi_1(\Ham(M,\omega))$ and a point $x \in M$, denote the capped orbit of $t \mapsto \psi ^t(x)$ with a capping $w$ by $[\psi ^t(x),w]$. We define its Maslov index $\mu([\psi ^t(x),w])$ via the trivialization of $w^\ast TM$ and the loop of symplectic matrices $d\psi ^t(x):T_xM\to T_{\psi^t (x)}M.$ The definition of Maslov indices cannot be directly applied to periodic orbits of a Hamiltonian $H$ since given a periodic orbit $[\phi_H ^t (x),w],$ $d\phi_H ^t(x):T_xM\to T_{\phi_H ^t (x)}M$ might not define a loop. To overcome this difficulty, Conley-Zehnder modified the definition of the Maslov index and introduced the Conley-Zehnder index 
$$\mu_{CZ}: \{A:[0,1]\to Sp(2n)| \det (A(1)- \id) \neq 0 \}\to \mathbb{Z}$$
which maps paths of symplectic matrices to integers. Thus, as in the case of Maslov indices, we define the Conley-Zehnder index of a non-degenerate periodic orbit of a Hamiltonian $H$ $[\phi_H ^t (x),w]$ (i.e. $d\phi_H(x) : T_x M\to T_x M $ has no eigenvalue which equals to $1$), denoted by $\mu_{CZ}([\phi_H ^t (x),w]),$ via the trivialization of $w^\ast TM$ and the path of symplectic matrices $d\phi_H ^t(x):T_xM\to T_{\phi_H ^t (x)}M$.

The following elementary properties are often used to calculate the action. 

\begin{prop}
Let $(M,\omega)$ be a symplectic manifold. Assume the Hamiltonian paths generated by $H$ and $G$ are homotopic rel. end points

i.e. there exists $ W:[0,1]\times [0,1]\to \Ham(M,\omega)$ such that 
\begin{enumerate}
\item $W(0,t)=\phi_H ^t,\ W (1,t)=\phi_G ^t.$
\item $W(s,0)=\id,\ W (s,1)=\phi_H=\phi_G.$ 
\end{enumerate}
Let $x\in \Fix(\phi_H)=\Fix(\phi_G)$ and $w$ be a capping of the orbit $\phi_H ^t(x)$. Then the action of the capped orbit $[\phi_G ^t(x),w']$ where $w':=w\# W $ ($W$ glued to $x$ along $\phi_H ^t(x)$) coincides with the action of $[\phi_H ^t(x),w]$:
$$\mathcal{A}_H([\phi_H ^t(x),w])=\mathcal{A}_G([\phi_G ^t(x),w'])$$
\end{prop}

\begin{prop}
Let $(M,\omega)$ be a symplectic manifold. 
\begin{enumerate}
\item For any Hamiltonian $H\in C^{\infty}(\mathbb{R}/\mathbb{Z}\times M,\mathbb{R})$,
$$\overline{H}(t,x):=-H(t,\phi_H ^t(x))$$ generates the Hamiltonian flow 
$$t \mapsto (\phi_H ^t) ^{-1}$$
and has the time-1 map $\phi_H ^{-1}.$
\item For any Hamiltonian $H\in C^{\infty}(\mathbb{R}/\mathbb{Z}\times M,\mathbb{R})$,
$$\tilde{H}(t,x):=-H(-t,x)$$ 
generates the Hamiltonian flow 
$$t \mapsto \phi_H ^{-t}$$
and has the time-1 map $\phi_H ^{-1}.$
\item Hamiltonian paths generated by $\overline{H}$ and $\tilde{H}$ are homotopic rel. end points.
\end{enumerate}
\end{prop}

\begin{prop}\label{useful prop}
Let $(M,\omega)$ be a symplectic manifold. 
\begin{enumerate}
\item For any Hamiltonians $H,G\in C^{\infty}(\mathbb{R}/\mathbb{Z}\times M,\mathbb{R})$,
$$H\#G(t,x):=H(t,x)+G(t,(\phi_H ^t)^{-1}(x))$$ 
generates the Hamiltonian flow 
$$t \mapsto \phi_H ^{t} \circ \phi_G ^t$$
and has the time-1 map $\phi_H\circ \phi_G.$

\item For any Hamiltonians $H,G\in C^{\infty}(\mathbb{R}/\mathbb{Z}\times M,\mathbb{R})$,
$$H\wedge G(t,x):=\begin{cases}
2 \chi' (2t) G(\chi(2t) ,x)\ (0\leq t\leq 1/2)\\
2 \chi' (2t-1) H(\chi (2t-1),\phi_G(x))\ (1/2\leq t\leq 1)
\end{cases}$$ 
generates the Hamiltonian flow
$$t \mapsto \begin{cases}
\phi_G ^{ \chi (2t)}\ (0\leq t\leq 1/2)\\
\phi_H ^{\chi(2t-1)}\circ \phi_G \ (1/2\leq t\leq 1)
\end{cases}$$
and has the time-1 map $\phi_H\circ \phi_G.$ Here,  
$$\chi : [0,1] \to [0,1],\ t \mapsto \chi (t)$$
 is a smooth map that is identically $0$ around $t=0$ and $1$ around $t=1$ and plays the role of gluing two Hamiltonians smoothly.

\item Hamiltonian paths generated by $H\#G$ and $H\wedge G$ are homotopic rel. end points (for any choice of $\chi : [0,1] \to [0,1],\ t \mapsto \chi (t)$).
\end{enumerate}
\end{prop}

The following two propositions will be used in Section \ref{section neg mono}. Proofs will be omitted as they follow from elementary arguments.

\begin{prop}\label{homotopy invariance}
Let $(M,\omega)$ be a symplectic manifold, $U$ a simply connected non-empty open set and $H$ a Hamiltonian such that $\phi_H(p)=p$ for all $p\in U$. Take any $x_0\in U$ and a capping $w_0:D^2\to M$ of the orbit $\phi_H ^t(x_0)$ and fix them.

For any $x\in U$, define a capping $w_x:D^2\to M$ of the orbit $\phi_H ^t(x)$ by $$w_x(s e^{2\pi i t}):=\phi_H ^t (c(s))\# w_0$$ where $c:[0,1]\to M$ is a smooth path from $x_0$ to $x$ and $\phi_H ^t (c(s))\# w_0$ denotes the gluing of $\phi_H ^t (c(s))$ and $w_0$ along $\phi_H ^t(x_0)$. Then we have the following:
\begin{enumerate}
\item $\mathcal{A}_{H}([\phi_H ^t(x),w_x])=\mathcal{A}_{H}([\phi_H ^t(x_0),w_0]).$
\item $\mu([\phi_H ^t(x),w_x])=\mu([\phi_H ^t(x_0),w_0]).$
\end{enumerate}
\end{prop}

\begin{prop}\label{opposite direction}
Let $(M,\omega)$ be a symplectic manifold, $H$ a Hamiltonian and $[\phi_H ^t(x),w]$ any capped 1-periodic orbit of $H$. Then 
\begin{enumerate}
\item $\overline{w}:D^2\to M,\ \overline{w}(se^{2\pi i t}):=w(se^{2\pi i (-t)})$ is a capping of the orbit $\phi_H ^{-t}(x)$
\item $\mu([\phi_H ^t(x),w])=-\mu([\phi_H ^{-t}(x),\overline{w}])$
\item $\mathcal{A}_{H}([\phi_H ^t(x),w])=-\mathcal{A}_{\tilde{H}}([\phi_{\tilde{H}} ^t(x),\overline{w}])$ where $\tilde{H}(t,x)=-H(-t,x)$.
\end{enumerate}
\end{prop}

\subsection{Hamiltonian Floer theory}
We fix a ground field $\K$ of zero characteristic (see Remark \ref{ground field}). We say that a Hamiltonian $H$ is non-degenerate if the diagonal set $\Delta:=\{(x,x)\in M\times M\}$ intersects transversally the graph of $\phi,$ $\Gamma_{\phi}:=\{(x,\phi(x))\in M\times M)\}$. For a non-degenerate $H$, we define the Floer chain complex $CF_\ast(H)$ as follows: 
$$CF_\ast(H):=\{ \sum_{ \widetilde{z} \in \widetilde{\mathcal{P}}(H)} a_{\widetilde{z} } \cdot  \widetilde{z}  : a_{\widetilde{z} } \in \K, (\forall \tau \in \mathbb{R},\ \#\{\widetilde{z} : a_{\widetilde{z} } \neq 0, \mathcal{A}_H(\widetilde{z}) \leq \tau\}<+\infty ) \}.$$

The Floer chain complex has a $\mathbb{Z}$-grading by the Conley-Zehnder index $\mu_{CZ}$. The boundary map counts certain solutions of a perturbed Cauchy-Riemann equation for a chosen $\omega$-compatible almost complex structure $J$ on $TM$, which can be viewed as isolated negative gradient flow lines of $\mathcal{A}_H$. This gives us a chain complex $(CF_{\ast} (H),\partial)$ called the Floer chain complex. Its homology is called the Floer homology of $(H,J)$ and is denoted by $HF_{\ast} (H,J)$. Often it is abbreviated to $HF_{\ast} (H)$ as Floer homology does not depend on the choice of an almost complex structure.

Recapping of a capped orbit by $A\in \pi_2(M)$ changes the action and the Conley-Zehnder index as follows:
\begin{itemize}
\item $\mathcal{A}_H([z,w\# A])=\mathcal{A}_H([z,w])-\omega(A).$
\item $\mu_{CZ}([z,w\# A])=\mu_{CZ}([z,w])-2c_1(A).$
\end{itemize}

We define the filtered Floer complex of $H$ by 
$$CF_{\ast} ^{\tau}(H):=\{\sum a_z z\in CF_{\ast}(H):\mathcal{A}_H(z)<\tau \}.$$
Since the Floer boundary map decreases the action, $(CF_{\ast} ^{\tau}(H),\partial)$ forms a chain complex. The filtered Floer homology of $H$ which is denoted by $HF_{\ast} ^{\tau}(H)$ is the homology defined by the chain complex $(CF_{\ast} ^{\tau}(H),\partial)$.

It is useful to clarify our convention of the Conley-Zehnder index since conventions change according to literature. We fix our convention by requiring that for a $C^2$-small Morse function $f:M\to \mathbb{R}$, each critical point $x$ of $f$ satisfy
$$\mu_{CZ}([x,w_x])=i(x)$$
where $i$ denotes the Morse index and $w_x$ is the trivial capping.

\subsection{Quantum homology and Seidel elements}\label{QH}

We sketch some basic definitions and properties concerning the quantum homology. Once again, we fix a ground field $\K$. For further details of the concepts sketched in this section, we refer to \cite{[MS04]}.

Let $(M,\omega)$ be a closed symplectic manifold. Define
$$\Gamma:=\pi_2(M)/(\Ker(\omega)\cap \Ker(c_1)).$$

The Novikov ring $\Lambda$ is defined by
$$\Lambda:=\{\sum_{A\in \Gamma} a_A \otimes e^A: a_A \in \K, (\forall \tau \in \mathbb{R}, \#\{a_A\neq 0,\omega(A)<\tau\}<\infty ) \}.$$

The quantum homology of $(M,\omega)$ is defined by
$$QH_\ast(M;\K):=H_\ast(M;\K) \otimes_{\K} \Lambda.$$
The quantum homology has a ring structure with respect to the quantum product denoted by $\ast$. It is defined as follows:
$$\forall a,b,c \in H_\ast (M;\K),\ (a\ast b)\circ c:=\sum_{A\in \Gamma} GW_{3,A}(a,b,c)\otimes e^A$$
where $\circ $ denotes the intersection index and $GW_{3,A}$ denotes the 3-pointed Gromov-Witten invariant in the class $A$. See \cite{[MS04]} for details.

When $(M,\omega)$ is either monotone or negative monotone, then the quantum homology ring $QH_\ast(M;\K)$ can be expressed in a simple way using the field of Laurent series. We first explain the case of monotone symplectic manifolds. In this case, $\Gamma \simeq \mathbb{Z}$ with a generator $A$ such that 
$$\omega(A)=\lambda_0,\ c_1(A)= N_M.$$

Thus the Novikov ring $\Lambda$ is the ring of formal Laurent series 
$$\K [| s^{-1}, s ]:=\{\sum_{k \geq k_0} a_k s^k :a_j \in \K , k_0 \in \Z \}$$
where $s:=e^{A}$ and the quantum homology ring $QH_\ast(M;\K)$ is
$$QH_\ast(M;\K)= H_\ast(M;\K) \otimes_\K \K[| s^{-1}, s ].$$

The quantum product is expressed by
$$\forall a,b \in H_\ast (M;\K),\ a\ast b=a \cap b + \sum_{k>0}(a\ast b)_k \cdot s^{k}.$$
The series on the right hand side runs over only non-positive powers since the elements of $\Gamma$ appearing in the sum represents pseudo-holomophic spheres and pseudo-holomophic spheres has non-negative $\omega$-area (remember that $s$ represents a sphere $A$ such that $\omega(A)=\lambda_0$).

When $(M,\omega)$ is negative monotone, by denoting the generator $A$ of $\Gamma$ which satisfies
$$\omega(A) = +\lambda_0,\ c_1(A) = -N_M$$
and by denoting $s:=e^A$, we have
$$QH_\ast(M;\K)= H_\ast(M;\K) \otimes_\K \K[| s^{-1}, s ]$$
just as in the monotone case.

\begin{example}
The quantum homology ring of $(\mathbb{C}P^n,\omega_{FS})$ is expressed as follows:
$$QH_\ast (\mathbb{C}P^n; \K )=\frac{ \K[|s^{-1},s][h]}{\langle h^{\ast (n+1)} =[\mathbb{C}P^{n}]\cdot s^{-1}  \rangle}$$
where $h \in H_{2n-2}(\mathbb{C}P^n ; \K)$ denotes the projective hyperplane class, $s$ denotes the generator of the Novikov ring and $h^{\ast (n+1)}:=\underbrace{h\ast h \ast \cdots \ast h}_{n+1-times}$.
\end{example}

There is a canonical isomorphism called the PSS-isomorphism between Floer homology and quantum homology which will be denoted by $\Phi$:
$$\Phi_{PSS,H;\K}:QH_\ast(M;\K)\stackrel{\sim}{\longrightarrow} HF_\ast(H).$$
PSS-isomorphism preserves the ring structure: for $a,b\in QH_\ast(M;\K)$,
$$\Phi_{PSS,H;\K} (a)\ast_{pp} \Phi_{PSS,H;\K}(b)=\Phi_{PSS,H;\K} (a\ast b)$$
where $\ast_{pp}$ denotes the pair-of-pants product.

Next, we briefly recall the definition of the Seidel element. The idea goes back to Seidel \cite{[Sei97]}. For a Hamiltonian loop $\psi \in \pi_1(\Ham(M,\omega))$, we can define a Hamiltonian fiber bundle
$$(M,\omega) \hookrightarrow (M_\psi, \Omega_\psi) \to (S^2,\omega_{area})$$
where, unit disks $D^2 _j,\ j=1,2$, 
$$M_\psi:= (D^2 _1 \times M) \amalg (D^2 _2 \times M )/ \sim,$$
$$(z_1,x) \sim (z_2 , y) \Longleftrightarrow z_1=z_2=e^{2\pi i t} , y=\psi^t (x).$$
The form $ \Omega_\psi$ is a family of symplectic form on $TM_\psi ^{vert}=\Ker(d\pi)$ parametrized by points of $S^2$. We fix almost complex structures $j$ on $S^2$ and $J$ on $M_\psi$ such that $d\pi$ is pseudo-holomorphic i.e. $j \circ d\pi = d\pi \circ J$ and for every $z\in S^2$, $J|_{\pi^{-1}(z)}$ defines a $\Omega_{\psi}$-compatible almost complex structure on $M_\psi$. For a section class $\sigma \in \pi_2(M_\psi)$, we denote the set of $(j,J)$-pseudo-holomorphic spheres in the class $\sigma$ by $SecCl(j,J,\sigma)$. The image of $SecCl(j,J,\sigma)$ by the evaluation map $ev: S^2 \to M$ at $z _0 \in S^2$ defines a homology class $[ev_{z_0}(SecCl(j,J,\sigma))]$ of $M$. We thus define the Seidel element $\S_{\psi,\sigma} \in QH_\ast(M;\K)$ by
$$\S_{\psi,\sigma} := \sum_{A \in \Gamma} [ev_{z_0}(SecCl(j,J,\sigma+A))] \otimes e^A.$$

\subsection{Hamiltonian spectral invariants}\label{spec inv}
Let $i^{\tau}: CF_{\ast} ^{\tau}(H)\to CF_{\ast} (H)$ be the natural inclusion map and denote by $i_{\ast} ^{\tau}:HF_{\ast} ^{\tau}(H)\to HF_{\ast} (H)$ the induced map on homology. For a quantum homology class $a\in QH_{\ast}(M;\K),$ define the spectral invariant by
$$c(H,a)=c(H,a;\K):=\inf \{\tau \in \mathbb{R}:\Phi_{PSS,H;\K}(a)\in \Im(i_{\ast} ^{\tau})\}.$$
Spectral invariants were introduced by Viterbo \cite{[Vit92]} in terms of generating functions and later their counterparts in Floer theory were studied by Schwarz for aspherical symplectic manifolds \cite{[Sch00]} and Oh for closed symplectic manifolds \cite{[Oh05]}. We list some basic properties of spectral invariants.

\begin{prop}\label{prop spec inv}
Spectral invariants satisfy the following properties where $H,G$ are Hamiltonians:
\begin{enumerate}
\item For any $a\in QH_{\ast}(M;\K) \backslash \{0\}$,
$$\mathcal{E}^- (H-G)\leq c(H,a)-c(G,a)\leq \mathcal{E}^+(H-G)$$ where $$\mathcal{E}^-(H):=\int_{t= 0} ^1 \inf_x {H_t(x)}dt,\ \mathcal{E}^+(H):=\int_{t=0} ^1 \sup_x{H_t(x)} dt,$$
$$\mathcal{E}(H):=\mathcal{E}^+(H)-\mathcal{E}^-(H).$$
\item For any $a\in QH_{\ast}(M;\K) \backslash \{0\}$,
$$c(H,a) \in \Spec (H)$$
\begin{itemize}
\item for any {\ham} $H$ when $(M,\omega)$ is rational.
\item for any non-degenerate {\ham} $H$ when $(M,\omega)$ is a general closed symplectic manifold.
\end{itemize}

\item For any $a,b\in QH_{\ast}(M;\K) \backslash \{0\}$,
$$c(H\#G,a\ast b)\leq c(H,a)+c(G,b).$$

\item Let $U$ be a non-empty subset of $M$. 
$$c(H,[M])\leq e(\Supp(H)):=\inf \{\mathcal{E}(G):\phi_{G}(\Supp(H))\cap \Supp(H)=\emptyset\}.$$

\item Let $f:M\to \mathbb{R}$ be an autonomous Hamiltonian and $a \in H_\ast (M;\K)$. For a sufficiently small $\varepsilon>0$, we have 
$$c(\varepsilon f,a)=c_{LS}(\varepsilon f,a)=\varepsilon \cdot c_{LS}(f,a)$$
where $c_{LS}(f,a)$ is the Lusternik--Schnirelmann spectral invariant defined by 
$$c_{LS}(f,a):=\inf \{\tau:a\in \Im(H_{\ast}(\{f\leq \tau\})\to H_{\ast}(M)) \}.$$

\item For $\psi \in \pi_1(\Ham(M,\omega))$, a section class $\sigma$ of the Hamiltonian fiber bundle $M_{\psi} \to S^2$, and $ a\in QH_\ast(M;\K) \backslash \{0\}$ we have
$$c(\psi^\ast H,a)=c(H,\mathcal{S}_{\psi,\sigma}\ast a)+const(\psi,\sigma)$$
where 
$$(\psi^\ast H)_t:=(H_t-K_t)\circ \psi^t,\ \psi^t:=\phi_K ^t,\ \phi_{\psi^\ast H} ^t=(\psi^t)^{-1}\circ \phi_H ^t$$
and $const(\psi,\sigma)$ denotes a constant depending on $\sigma$ and $K$.

\item For any $\psi \in \Symp_0 (M,\omega)$ and $a\in QH_{\ast}(M;\K) \backslash \{0\}$,
$$c(H\circ \psi , a) = c(H, a).$$

\end{enumerate}
\end{prop}

\begin{remark}
\begin{enumerate}
\item For a set $A,$ $e(A):=\inf \{\mathcal{E}(G):\phi_{G}(A)\cap A=\emptyset\} $ is called the displacement energy of $A$.
\item Strictly speaking, spectral invariants $c(H,\cdot)$ can be defined only if $H$ is non-degenerate since they are defined via Floer homology of $H$. However, by Proposition \ref{prop spec inv} (1), one can define $c(H,\cdot)$ for a degenerate Hamiltonian $H$ by considering an approximation of $H$ by non-degenerate Hamiltonians.

\end{enumerate}
\end{remark}

The spectral norm of $H$ is defined by 
$$\gamma(H):=c(H,[M])+c(\overline{H},[M])$$
where $[M]$  denotes the fundamental class. We also define a spectral norm for Hamiltonian diffeomorphisms by
$$\gamma: \Ham(M,\omega) \to \R$$ 
$$\gamma(\phi):=\inf_{\phi_H=\phi}\gamma(H).$$

\begin{remark}\label{ground field}
In this paper, we work with a fixed ground field $\K$ of zero characteristic but some results hold for other ground rings. More precisely, Theorem \ref{C0-conti of gamma:CPn} and \ref{monotone bound extra} hold for spectral norms respectively with any ground field $\K$ and with any ground ring $R$ that is commutative and unital e.g. $\Z$. In fact, Usher proved in \cite{[Ush08]} that whenever one can define a Floer chain complex with a ground ring that is Noetherian, spectral invariants can be defined as above and satisfy properties listed in Proposition \ref{prop spec inv}. For weakly-monotone symplectic manifolds, one can define Floer chain complexes with any ground field $\K$ \cite{[MS04]}. Moreover, for monotone symplectic manifolds, one can define Floer chain complexes with any ground ring $R$ that is commutative and unital \cite{[LZ18]}, \cite{[Zap15]}. For general closed symplectic manifolds where one needs to use virtual cycle techniques in order to build Floer chain complexes \cite{[FO99]}, \cite{[FOOO09]}, \cite{[LT98]}, the ground field $\K$ should have zero characteristic.  

\end{remark}

\subsection{Barcodes}\label{barcodes}
In this subsection, we roughly explain how to define barcodes for Hamiltonian diffeomorphisms on (negative) monotone symplectic manifolds following \cite{[LSV18]}. We also refer to \cite{[PSS17]} and \cite{[UZ16]} for constructions of barcodes in symplectic topology.

A finite barcode is a finite set of intervals 
$$B=\{I_j=(a_j,b_j]  : a_j \in \R, b_j \in \R \cup \{+\infty\} \}_{1\leq j\leq N}.$$ 

Two finite barcodes $B,B'$ are said to be $\delta$-matched if, after deleting some intervals of length less than $2\delta$, there exists a bijective matching between the intervals of $B$ and $B'$ such that the endpoints of the matched intervals are less than $\delta$ of each other. The bottleneck distance of $B,B'$ is defined as follows: 
$$d_{bot}(B,B'):=\inf \{\delta>0:  B\ and\ B'\ are\ \delta-matched \}.$$

Barcodes of non-degenerate Hamiltonian diffeomorphisms were first defined in \cite{[PS16]} for symplectic manifolds that are symplectically aspherical via filtered Floer homology. For symplectically non-aspherical manifolds, (filtered) Floer homology groups do not satisfy the "finiteness" condition and in order to overcome this issue, \cite{[PSS17]} defines barcodes for non-degenerate {\hamd}s on monotone symplectic manifolds by fixing a degree. Later, in order to define barcodes of degenerate Hamiltonian diffeomorphisms on spheres, Le Roux--Seyfaddini--Viterbo \cite{[LSV18]} considered a completion of the set of finite barcodes with respect to the bottleneck distance. Their method applies more generally to (negative) monotone symplectic manifolds. We will briefly review their idea.

Let $Barcodes$ denote the set of a collection of intervals $B=\{I_j\}_{j\in \mathbb{N}}$ such that for any $\delta>0$ only finitely many of the intervals $I_j$ have lengths greater than $\delta$. The bottleneck distance $d_{bot}$ extends to $Barcodes$. The space $(Barcodes , d_{bot})$ is indeed the completion of the space of finite barcodes. Given a barcode $B=\{I_j\}_{j\in \mathbb{N}}$ and $c\in \mathbb{R}$, define $B+c = \{I_j +c\}_{j\in \mathbb{N}}$, where $I_j +c$ is the interval obtained by adding $c$ to the endpoints of $I_j$. Define an equivalence relation $\sim$ by $B\sim B'$ if $B' =B+c$ for some $c \in \mathbb{R}$. We will denote the quotient space of $Barcodes$ with the relation $\sim$ by $\widehat{Barcodes}$.

We explain briefly how to map a (possibly degenerate) Hamiltonian diffeomorphism on a (negative) monotone symplectic manifold to a barcode following \cite{[LSV18]}. Given a non-degenerate Hamiltonian $H$ and an integer $k\in \mathbb{Z}$, the filtered $k$-th Floer homology group $\{HF_k ^{\tau}(H)\}_{\tau \in \mathbb{R}}$ forms a persistence module. For this filtered vector spaces, one can define a barcode in the same way as in \cite{[PS16]} and we denote the barcode by $B_k(H)$. We define the barcode of $H$ by 
$$B(H):=\sqcup_{k}  B_k(H) \in Barcodes.$$ 
For two Hamiltonians $H,G$ such that $\phi_H=\phi_G$, their Floer homology groups coincide up to shifts of index and action filtration i.e. $HF_{\ast}^{\tau}(H)\simeq HF_{\ast+k_0}^{\tau+\tau_0}(G)$ for some $k_0\in \mathbb{Z},\tau_0\in \mathbb{R}$ (see, for example, \cite{[MS04]} Section 12.5). Thus $B(H)=B(G)$ and therefore, we define the barcode map $B$ as follows:
$$B : \Ham(M,\omega) \to \widehat{Barcodes}.$$
$$B(\phi):=B(H)$$
for any $H$ such that $\phi_H=\phi$.

Kislev-Shelukhin \cite{[KS18]} proved the following inequality to estimate the bottleneck distance between  barcodes of $\phi,\psi\in \Ham(M,\omega)$:
\begin{equation}\label{KS ineq}
d_{bot}(B(\phi),B(\psi))\leq \frac{1}{2}\gamma(\phi^{-1}\psi).
\end{equation}

This implies that once we obtain the $C^0$-continuity of $\gamma$, the map
$$B : (\Ham(M,\omega),d_{C^0}) \to (\widehat{Barcodes},d_{bot})$$
is continuous. Thus, Corollary \ref{neg mono barcode} is a direct consequence of Theorem \ref{neg mono}.

\section{Proofs}
In this section, we prove the results claimed in the introduction. We start from the case of negative monotone symplectic manifolds since the proof is based on a similar idea to the case of rational symplectic manifolds but it is simpler.

\subsection{Proofs of Theorem \ref{monotone bound extra} and \ref{neg mono}}\label{section neg mono}
We prove Theorem \ref{monotone bound extra} (1) and Theorem \ref{neg mono}. It is achieved by combining the following Propositions \ref{prop1} and \ref{lem2}.

\begin{prop}\label{prop1}
Let $(M,\omega)$ be a monotone or negative monotone symplectic manifold, $U$ be a simply connected open subset of $M$. For any $\varepsilon>0$, there exists $\delta>0$ such that if $H\in C^{\infty}(\mathbb{R}/\mathbb{Z}\times M,\mathbb{R})$ satisfies $d_{C^0}(\id,\phi_H)<\delta$ and $\phi_H(x)=x$ for all $x\in U$, then 
\begin{itemize}
\item when $(M,\omega)$ is monotone, 
$$\gamma(H)<\frac{\dim (M)}{N_M} \lambda_0 + \varepsilon.$$

\item when $(M,\omega)$ is negative monotone, 
$$\gamma(H)<\varepsilon.$$
\end{itemize}
\end{prop}

\begin{prop}\label{lem2}$($\cite{[BHS18b]} Lemma 4.2$)$

Let $(M,\omega)$ be any closed symplectic manifold. For any $\varepsilon>0$, there exists a non-empty open ball $B \subset M$ satisfying the following properties: its displacement energy is estimated by $e(B)<\eps$ and for any $\varepsilon'>0$, there exists $\delta'>0$ such that if $\phi_H \in \Ham(M,\omega),d_{C^0}(\id_M,\phi_H)<\delta'$, then there exist $G\in C^{\infty}(\mathbb{R}/\mathbb{Z}\times M \times M,\mathbb{R})$ such that 
\begin{enumerate}
\item $\gamma(G)<\varepsilon$
\item $d_{C^0}(\id_{M\times M},\phi_G)<\varepsilon'$
\item $(\phi_H \times \phi_H ^{-1})\circ \phi_G|_{B\times B}=\id_{B\times B}$
\end{enumerate}
\end{prop}

Proposition \ref{lem2}, proven by Buhovsky-Humili\`ere-Seyfaddini \cite{[BHS18b]}, claims that given a Hamiltonian diffeomorphism $\phi$ on $M$, one can always deform  the Hamiltonian diffeomorphism $\phi \times \phi^{-1}$ to a Hamiltonian diffeomorphism on $M\times M$ that does not move any point on a certain open set by composing with a both $C^0$- and $\gamma$-small Hamiltonian diffeomorphism on $M\times M$.

We postpone the proof of Proposition \ref{prop1} and first briefly review the proof of Proposition \ref{lem2} due to  \cite{[BHS18b]} as we will need some parts of the proof in the proof of Claim \ref{step 1}.

\begin{proof}(of Proposition \ref{lem2} by \cite{[BHS18b]})

Let $\eps>0$ and fix any non-empty open ball $B'$ whose displacement energy satisfies $e(B') < \eps /4$.

\begin{claim}\label{def of Q}$($Claim 4.3 \cite{[BHS18b]}$)$

There exists a Hamiltonian $Q$ on $M\times M$ and an open ball $B''$ in $M$ such that
\begin{itemize}
\item $Supp(Q) \subset B' \times B'$.
\item $\forall (x,y ) \in B''\times B'',\ \phi_{Q} (x,y) = (y,x).$
\item  Denote the origin point of the ball $B'$ by $x_0$. The point $(x_0,x_0)$ is fixed by the flow of $Q$: $\forall t, \phi_Q ^t ((x_0,x_0)) = (x_0,x_0).$
\end{itemize}
\end{claim}

Now, define
\begin{equation}\label{def of G}
G:= (0 \oplus \ovl{H}) \# \ovl{Q} \# (0 \oplus H) \# Q. 
\end{equation}
This Hamiltonian $G$ satisfies the following:

$$ c(G,[M\times M]) = c((0\oplus \ovl{H})\# \overline{Q}\# (0\oplus H)\#Q,[M\times M])$$
$$\leq c(( \overline{0\oplus H})\# \overline{Q}\# (0\oplus H),[M\times M])+c(Q,[M\times M])$$
$$=c(\overline{Q},[M\times M])+c(Q,[M\times M])$$
$$\leq e(Supp (\overline{Q}))+e(Supp (Q)) \leq 2e(B' \times B')$$
and as the ball $B'$ was chosen so that $e(B')< \eps/4$, we get

\begin{equation}\label{step 1 estimate}
 c(G,[M\times M]) <\varepsilon/2 .
\end{equation}

 We can estimate $c(\ovl{G},[M\times M])$ in the same way and we get
$$\gamma (G) <\eps .$$
Now, let $B$ be a ball whose closure is included in $B''$ and make sure that the origin of $B$ is the same as the origin of $B'$, namely $x_0$. If we require $\phi_H$ to be $C^0$-close enough to $\id$ so that $\phi_H (B) \subset B''$, then for all $(x,y ) \in B\times B$, we have 
$$(\phi_H \times \phi_H ^{-1})\circ \phi_G (x,y) = (x,y).$$
This finishes the proof of Proposition \ref{lem2}.
\end{proof}

Now, before proving Proposition \ref{prop1}, we prove Theorem \ref{monotone bound extra} (1) and Theorem \ref{neg mono}.

\begin{proof}(of Theorem \ref{monotone bound extra} (1) and Theorem \ref{neg mono})

Note that if $(M,\omega)$ is (negative) monotone, then so is $(M\times M,\omega \oplus \omega)$. Given any $\varepsilon>0$, we can take a ball $B$ in $M$ as in Proposition \ref{lem2}. By Proposition \ref{lem2}, for any $\eps'>0$, there exists $\delta'>0$ such that if $d_{C^0}(\id_M,\phi)<\delta'$, then there exist $G\in C^{\infty}(\mathbb{R}/\mathbb{Z}\times M \times M,\mathbb{R})$ such that 
\begin{enumerate}
\item $\gamma(G)<\varepsilon$
\item $d_{C^0}(\id_{M\times M},\phi_G)<\eps'$
\item $[(\phi \times \phi^{-1})\circ \phi_G]|_{B\times B}=\id_{B\times B}$
\end{enumerate}

We take $\eps'>0$ small enough so that
$$d_{C^0}(\id_{M\times M}, (\phi \times \phi^{-1})\circ \phi_G) < \delta$$
is satisfied where $\delta>0$ is a positive number as in Proposition \ref{prop1} which is determined by $B\times B$ and $\eps>0$. This is achievable as
$$d_{C^0}(\id_{M\times M}, (\phi \times \phi^{-1})\circ \phi_G) \leq d_{C^0}(\id_{M\times M}, \phi_G)+d_{C^0}(\phi_G, (\phi \times \phi^{-1})\circ \phi_G)$$
$$= d_{C^0}(\id_{M\times M}, \phi_G)+d_{C^0}(\id_{M\times M}, \phi \times \phi^{-1})\leq \eps'+2 \delta'.$$

Now, take any Hamiltonian $H$ generating $\phi$: $\phi_H=\phi$. Then $(H\oplus \overline{H})\# G$ generates $(\phi \times \phi^{-1})\circ \phi_G$ so if  by Proposition \ref{prop1}, we have 
\begin{itemize}
\item if $(M,\omega)$ is monotone, then
$$\gamma((H\oplus \overline{H})\# G)<\frac{ \dim (M \times M)}{ N_{M\times M}}  \lambda_0+ \varepsilon=2\cdot \frac{  \dim (M)}{ N_{M}}  \lambda_0 + \varepsilon.$$
\item if $(M,\omega)$ is negative monotone, 
$$\gamma((H\oplus \overline{H})\# G)<\varepsilon.$$
\end{itemize}

As $\gamma(H\oplus \overline{H})=2\gamma(H)$ (by Theorem 5.1. in [EP09]), we have
$$2 \gamma(H) = \gamma(H\oplus \overline{H})\leq \gamma((H\oplus \overline{H})\# G)+\gamma(\overline{G})$$
$$=\gamma((H\oplus \overline{H})\# G)+\gamma(G)<\gamma((H\oplus \overline{H})\# G)+\varepsilon.$$
Therefore, 
\begin{itemize}

\item if $(M,\omega)$ is monotone, then 
$$2 \gamma(H)< 2\cdot \frac{  \dim (M)}{ N_{M}} \lambda_0 + \varepsilon + \varepsilon,$$
thus 
$$\gamma(H) < \frac{\dim (M)}{N_M} \lambda_0 + \eps.$$ 
This proves Theorem  \ref{monotone bound extra} (1). 
\item if $(M,\omega)$ is negative monotone, then $2\gamma(H)<2\varepsilon$, thus
$$\gamma(H) < \eps.$$ 
This proves Theorem \ref{neg mono}.
\end{itemize}

\end{proof}

\begin{proof}(of Theorem \ref{neg mono} (2))

Once we know that spectral norms are well-defined on $\Ham(M,\omega)$, the $C^0$-continuity at $\id$ follows directly from Theorem \ref{neg mono} (1). The $C^0$-continuity at $\phi \in \Ham(M,\omega)$ is a consequence of the triangle inequality: for any $\varepsilon>0$, if we take $d_{C^0}(\phi,\psi)$ small enough so that $d_{C^0}(\id,\phi^{-1}\circ \psi)<\delta$ where $\delta$ is taken as in Theorem \ref{neg mono} (1). Then,
$$|\gamma(\psi)-\gamma(\phi)|\leq \gamma(\phi^{-1}\circ \psi) < \varepsilon.$$
By using the $C^0$-continuity, we can define the spectral norm for Hamiltonian homeomorphisms in the following way: for $\phi \in \ovl{\Ham}(M,\omega)$, take a sequence $\phi_k \in \Ham (M,\omega)$ that $C^0$-converges to $\phi$. Define $\gamma(\phi):= \lim_{k \to +\infty} \gamma (\phi_k)$. Note that any approximating sequence will give the same limit. This completes the proof of Theorem \ref{neg mono} (2).
\end{proof}

We now prove Proposition \ref{prop1}.

\begin{proof}(of Proposition \ref{prop1})

Take a Morse function $f:M\to \mathbb{R}$ whose critical points are located in $U$. We assume that $f$ is $C^2$-small enough so that its Hamiltonian flow does not admit any non-constant periodic points and that $osc(f):=\max{f}-\min{f}<\varepsilon$. Since $\phi_f$ has no fixed points in $M\backslash U$, there exists $\delta>0$ such that $$\forall x\in M\backslash U,\ d(x,\phi_f(x))>\delta.$$

We will now see that if $\phi_H$ is $C^0$-close enough to $\id$, then
$$\Crit(f)=\Fix(\phi_H \circ \phi_f).$$
First, $\Crit(f)\subset \Fix(\phi_H \circ \phi_f)$ follows from $\forall x\in U, \phi_H(x)=x$. Next, we will see $\Fix(\phi_H \circ \phi_f) \subset \Crit(f)$ if $\phi_H$ is $C^0$-close enough to $\id$. Let $x\in \Fix(\phi_H \circ \phi_f).$
\begin{enumerate}
\item Assume $x\in U.$ Then, $\phi_f(x)=\phi_H \circ \phi_f(x)=x$ and since $\Crit(f) = \Fix (\phi_f)$, we have $x\in \Crit(f).$ 
\item Assume $x\notin U$. Then, $\phi_H(x) \notin U$ and
$$d_{C^0}(x,\phi_H \circ \phi_f(x))\geq d_{C^0}(\phi_f(x),x) - d_{C^0}(\phi_f(x),\phi_H\circ \phi_f (x))$$
$$\geq \delta -d_{C^0}(\id,\phi_H).$$
If we take $\phi_H$ to be $C^0$-close enough to $\id$ so that the last equation become positive, then $x\notin \Fix(\phi_H \circ \phi_f).$ Thus $x\in \Fix(\phi_H \circ \phi_f)$ implies $x\in U$ and $x=\phi_H \circ \phi_f (x)=\phi_f(x).$ Thus $x\in \Crit(f).$
\end{enumerate}
We have proven that if $\phi_H$ is $C^0$-close enough to $\id$, then
$$\Crit(f)=\Fix(\phi_H \circ \phi_f).$$
Thus, for such $\phi_H$ and for any $x \in \Crit(f)=\Fix(\phi_H\circ \phi_f)$, its orbit is $\phi_{H\# f} ^t(x)=\phi_H ^t (x)$ and thus,
 $$\Spec(H\# f)=\{f(x)+ \mathcal{A}_{H}([\phi_{H} ^t(x),w]):x\in \Crit(f),[\phi_{H} ^t(x),w]\in \Crit( \mathcal{A}_{H}) \}.$$

Take any $x_0\in \Crit(f)$ and a capping $w_0:D^2\to M$ of the orbit $\phi_H ^t(x_0)$ i.e. $w_0(e^{2\pi i t})=\phi_H ^t(x_0)$. We fix this capped orbit $[\phi_H ^t (x_0),w_0]$ in the sequel.

For any $x\in \Crit(f)$, define a capping $w_x:D^2\to M$ of the orbit $\phi_H ^t(x)$ by $$w_x(s e^{2\pi i t}):=\phi_H ^t (c(s))\# w_0$$ where $c:[0,1]\to U$ is a smooth path from $x_0$ to $x$ and $\phi_H ^t (c(s))\# w_0$ denotes the gluing of $\phi_H ^t (c(s))$ and $w_0$ along $\phi_H ^t(x_0)$.

Recall that $\gamma(H)=c(H,[M])+c(\overline{H},[M])$ and we will estimate $c(H,[M])$ and $c(\overline{H},[M])$ separately.

By the triangle inequality, 
$$c(H,[M])\leq c(H\# f,[M])+c(\overline{f},[M]).$$
For the second term we know that 
$$c(\overline{f},[M])=c(-f,[M])\leq \varepsilon$$
as $f$ is $C^2$-small and $osc(f)<\varepsilon$.

For the first term, 
$$c(H\# f,\cdot)\in \Spec(H\# f)$$
so there exists a point $x\in \Crit(f)$ and a sphere $A:S^2\to M$ such that
\begin{itemize}
\item $\mathcal{A}_{H\# f}([\phi_{H} ^t(x),w_x\#A])=c(H\# f,[M]).$
\item $\mu_{CZ}([\phi_{H} ^t(x),w_x\#A])=\deg ([M])=2n.$
\end{itemize}
The sphere $A$ plays the role of correcting the capping of the capped orbit $[\phi_H ^t(x),w_x]$ to achieve the appropriate capped orbit which realizes the spectral invariant $c(H\# f,[M])$.

The action and the index can be rewritten in the following way where $i$ denotes the Morse index:
\begin{itemize}
\item $\mathcal{A}_{H\# f}([\phi_{H} ^t(x),w_x\#A])=f(x)+\mathcal{A}_{H}([\phi_H ^t(x),w_x])-\omega(A).$
\item $\mu_{CZ}([\phi_{H} ^t(x),w_x\#A])=i(x)+2\mu([\phi_H ^t(x),w_x])-2c_1(A).$
\end{itemize}
Thus we get the following two equations.

\begin{subequations}
\begin{align}
c(H\# f,[M])=f(x)+\mathcal{A}_{H}([\phi_H ^t(x),w_x])-\omega(A).\label{action1} \\
2n=i(x)+2\mu([\phi_H ^t(x),w_x])-2c_1(A).\label{index1}
\end{align}
\end{subequations}

In the same way, there exist a point $y\in \Crit(f)$ and a sphere $B:S^2\to M$ such that

\begin{subequations}
\begin{align}
c(\overline{H}\# f,[M])=f(y)+\mathcal{A}_{\overline{H}}([\phi_{\overline{H}} ^t(y),\overline{w_y}])-\omega(B).\label{action2} \\
2n=i(y)+2\mu([\phi_{\overline{H}} ^t(y),\overline{w_y}])-2c_1(B).\label{index2}
\end{align}
\end{subequations}

Here, the capping $\overline{w_y}$ is $$\overline{w_y}(se^{2\pi i t}):=w_y(se^{2\pi i (-t)}).$$
Thus, by adding the equations \ref{action1} and \ref{action2}, we obtain
$$\gamma(H)\leq 2c(-f,[M])+c(H\# f,[M])+c(\overline{H}\# f,[M])$$
$$=2c(-f,[M])+f(x)+f(y)+\mathcal{A}_{H}([\phi_H ^t(x),w_x])+\mathcal{A}_{\overline{H}}([\phi_{\overline{H}} ^t(y),\overline{w_y}])-\omega(A+B)$$
$$\leq 4\varepsilon-\omega(A+B)$$
where Proposition \ref{homotopy invariance} and \ref{opposite direction} were used in the last line. In the same way, by adding the equalities \ref{index1} and \ref{index2}, we obtain 
$$4n=i(x)+i(y)+2\mu([\phi_H ^t(x),w_x])+2\mu([\phi_{\overline{H}} ^t(y),\overline{w_y}])-2c_1(A+B)$$
$$=i(x)+i(y)-2c_1(A+B).$$
Now, since $i(x),i(y)$ are Morse indices, we have $$0\leq i(x),i(y)\leq 2n=\dim (M)$$ and thus, 
$$0\leq 4n+2c_1(A+B)\leq 4n.$$ 
Thus, 
$$-2n \leq c_1(A+B)\leq 0.$$ 
Note that up to now, we have not used the (negative) monotonicity of $(M,\omega)$. Now,

\begin{itemize}
\item if $(M,\omega)$ is negative monotone, then 
$$-\omega(A+B)=-\lambda \cdot c_1(A+B) \leq 0.$$
\item if $(M,\omega)$ is monotone, then 
$$-\omega(A+B)=-\lambda \cdot c_1(A+B) \leq 2n \lambda = \frac{2n}{N_M} \lambda_0 .$$
\end{itemize}
Therefore, 
\begin{itemize}
\item if $(M,\omega)$ is negative monotone, then 
$$\gamma(H)\leq 4\varepsilon.$$
\item if $(M,\omega)$ is monotone, then 
$$\gamma(H)\leq  \frac{2n}{N_M} \lambda_0 +4\varepsilon .$$
\end{itemize}

 This completes the proof of Proposition \ref{prop1}.
\end{proof}

\subsection{Proof of Theorem \ref{alternative around an integer}}

The goal of this subsection is to prove Theorem \ref{alternative around an integer} which includes Theorem \ref{rational} as a special case. The argument is similar to the negative monotone case. We start by some additional definitions.

\begin{definition}
Let $(M,\omega)$ be any closed symplectic manifold and $a,b \in H_\ast (M;\K )\backslash \{0\}.$ We define the following:
$$\gamma_{a,b}:C^{\infty}(\mathbb{R}/\mathbb{Z}\times M,\mathbb{R})\to \mathbb{R},$$
$$\gamma_{a,b}(H):=c(H,a)+c(\overline{H},b).$$
\end{definition}

\begin{remark}
Of course, $\gamma_{[M],[M]}=\gamma$ where $\gamma$ is the usual spectral norm.
\end{remark}

\begin{theo}\label{alternative around an integer}
Let $(M,\omega)$ be a rational symplectic manifold and $a,b \in H_\ast (M;\K) \backslash \{0\}.$ For any $\varepsilon>0,$ there exists $\delta>0$ such that if $d_{C^0}(\id,\phi_H)<\delta,$ then 
$$|\gamma_{a,b}(H)-l\cdot \lambda_0|<\varepsilon$$
for some integer $l\in \mathbb{Z}$ depending on $a,b \in H_\ast (M;\K) \backslash \{0\}$ and $H$.
\end{theo}

Before proving Theorem \ref{alternative around an integer}, we will see the following consequence on the $C^0$-continuity of the spectral norm.

\begin{corol}\label{bound implies conti}
Let $(M,\omega)$ be a rational symplectic manifold. Assume that there exist constants $0<\kappa<1$ and $\delta'>0$ such that if $\phi \in \Ham(M,\omega),\ d_{C^0}(\id,\phi)\leq \delta'$, then $\gamma(\phi)\leq \kappa \cdot \lambda_0$. Then, $\gamma:\Ham(M,\omega) \to \R$ is $C^0$-continuous.
\end{corol}

Corollary \ref{bound implies conti} will be used to obtain the $C^0$-continuity of the spectral norm for $\C P^n$ in Theorem \ref{C0-conti of gamma:CPn}.

\begin{proof}(of Corollary \ref{bound implies conti})

It is enough to prove the continuity at $\id$ since $|\gamma(\phi)-\gamma(\psi)|\leq \gamma(\psi^{-1}\phi).$ For a given $\varepsilon\in (0,\frac{1}{2}(1-\kappa)\lambda_0)$, take $\delta>0$ as in Theorem \ref{rational}. Let $$\phi \in \Ham(M,\omega),\ d_{C^0}(\id,\phi)<\min\{\delta,\delta'\}.$$ There exists a Hamiltonian $H$ such that $\phi_H=\phi$ and 
$$\gamma(H)<\gamma(\phi)+\varepsilon< \kappa \cdot \lambda_0 +\frac{1}{2}(1-\kappa)\lambda_0$$
$$=\frac{1}{2}(1+\kappa)\lambda_0<\lambda_0-\varepsilon.$$
Thus, by Theorem \ref{rational}, $$\gamma(H)<\varepsilon.$$
Thus,
$$\gamma(\phi)\leq \gamma(H)<\varepsilon.$$
This implies the continuity of $\gamma$ at $\id$ and hence completes the proof of Corollary \ref{bound implies conti}.
\end{proof}

Now, we move to the proof of Theorem \ref{alternative around an integer}. The following Proposition will be needed.

\begin{prop}\label{prop:around an integer}
Let $(M,\omega)$ be a closed symplectic manifold. Fix an arbitrary point $x_0\in M$. There exists a constant $C>0$ satisfying the following property: For any point $x\in M$, there exists $\psi  \in \Ham(M,\omega)$ such that 
\begin{enumerate}
\item $\psi(x)=x_0$
\item $\|d\psi^{-1} \| \leq C$
\end{enumerate}
\end{prop}

The proof is elementary and thus will be omitted.

\begin{proof}(of Theorem \ref{alternative around an integer})

The proof is similar to the proof of Theorem \ref{neg mono}. For a given $\varepsilon>0$, we take a ball $B$ as in Proposition \ref{lem2}. We will denote the origin of the ball $B$ by $x_0$. For the open set $B\times B$, consider a Morse function $F:M\times M\to \mathbb{R}$ such that 
\begin{itemize}
\item $\Crit(F) \subset B\times B$.
\item $F$ is $C^2$-small enough so that $\Fix (\phi_F) = \Crit ( F )$ and that $osc(F):=\max{F}-\min{F}<\varepsilon$. 
\end{itemize} 
As $\phi_F$ has no fixed points in $M\backslash (B\times B)$, there exists $\delta>0$ such that for any $x\in M\times M \backslash (B\times B),\ d(x,\phi_F(x))>\delta.$

For any $\varepsilon'>0$, we can take $\delta'>0$ as in Proposition \ref{lem2}. By Proposition \ref{prop:around an integer}, for $x_0$, there exists a constant $C>0$ such that for any $x \in M$, there exists $\psi \in \Ham(M,\omega)$ such that
\begin{itemize}
\item $\psi(x)=x_0$
\item $\|d\psi^{-1} \|\leq C$ 
\end{itemize}
We consider $\phi_H$ so that $d_{C^0}(\id,\phi_H)<\delta'/C$. For any $x_\ast \in \Fix(\phi_H)$, we can take $\psi \in \Ham(M,\omega)$ such that $\psi(x_\ast )=x_0$ and $\|d\psi^{-1} \|  \leq C$.

Let $H':=H\circ \psi^{-1}$. We have 
$$d_{C^0}(\id,\phi_{H'})=d_{C^0}(\id,\psi^{-1}\phi_H\psi)=d_{C^0}(\psi^{-1},\psi^{-1}\phi_H)$$
$$\leq \|d\psi^{-1} \|d_{C^0}(\id,\phi_{H'})\leq C \cdot  \delta'/C=\delta'.$$

By Proposition \ref{lem2}, there exists $G\in C^{\infty}(\mathbb{R}/\mathbb{Z}\times M \times M)$ such that 
\begin{itemize}
\item $\gamma(G) < \varepsilon$.
\item $d_{C^0}(\id_{M\times M},\phi_G) < \varepsilon'$.
\item $(\phi_{H'} ^{-1} \times \phi_{H'}) \circ \phi_G|_{B\times B} = \id_{B\times B}$.
\end{itemize}
In addition, we have seen in the proof of Proposition \ref{lem2} that $G$ is defined by $G=(0\oplus \overline{H'})\# \overline{Q}\# (0\oplus H')\#Q$ where $Q$ is an autonomous Hamiltonian on $M \times M$ whose flow fixes the point $ (x_0,x_0)$ for all time $t$: $\phi_Q ^t ((x_0,x_0)) = (x_0,x_0).$ The spectral invariant of $G$ was estimated as
$$c(G,[M\times M]) < \frac{1}{2}\eps.$$
 All these properties of $G$ and $Q$ will be used in the following.

We will now split the proof into four steps.

$\bullet$ Step 1: The aim of this step is to prove the following:

\begin{claim}\label{step 1}
$$|c(\overline{H'}\oplus H',a \otimes b) - c((\overline{H'}\oplus H')\#G\#F,a \otimes b)| < \frac{3}{2} \eps.$$

\end{claim}

\begin{proof}
By the triangle inequality, we have

$$c((\overline{H'}\oplus H')\#G\#F,a \otimes b) - c(\overline{H'}\oplus H',a \otimes b) $$
$$\leq c(G\#F,[M\times M]) \leq  c(G,[M\times M])+c(F,[M\times M])<\frac{3}{2}\varepsilon.$$
Note that the final inequality uses,
$$c(F,[M\times M]) \leq \max(F) <\varepsilon$$
and the estimate 
$$c(G,[M\times M]) < \frac{1}{2}\eps.$$
The other side of the inequality follows from a similar estimate.
\end{proof}

$\bullet$ Step 2: The aim of this step is to prove the following:

\begin{claim}
$$c((\overline{H'}\oplus H')\#G\# F,a \otimes b)$$
$$=F(x,y)+\mathcal{A}_{(\overline{H'}\oplus H')\#G}([\phi_{(\overline{H'}\oplus H')\#G} ^t((x,y)), w_{x,y}] )+(\omega \oplus \omega)(A_1)$$
for some critical point $(x,y)$ of $F$, some capping $w_{x,y}$ and some $A_1 \in \pi_2(M \times M)$.
\end{claim}

\begin{proof}
As
$$d_{C^0}(\id,(\phi_{H'} ^{-1} \times \phi_{H'}) \circ \phi_G)\leq d_{C^0}(\id,\phi_G)+d_{C^0}(\phi_G,(\phi_{H'} ^{-1} \times \phi_{H'}) \circ \phi_G)$$
$$=d_{C^0}(\id,\phi_G)+d_{C^0}(\id,\phi_{H'} ^{-1} \times \phi_{H'})\leq \varepsilon'+\delta',$$ 
we can take $\varepsilon '>0$ small enough so that 
$$d_{C^0}(\id,(\phi_{H'} ^{-1} \times \phi_{H'}) \circ \phi_G) \leq \delta.$$

Therefore, as 
\begin{itemize}
\item for all $x \notin B\times B,\ d_{C^0}(x,\phi_F (x)) >\delta$,
\item $d_{C^0}(\id,(\phi_{H'} ^{-1} \times \phi_{H'}) \circ \phi_G)\leq \delta $,
\item $(\phi_{H'} ^{-1} \times \phi_{H'}) \circ \phi_G|_{B\times B} = \id_{B\times B}$,
\end{itemize}
we have $\Fix((\phi_{H'} ^{-1} \times \phi_{H'})\circ \phi_G \circ \phi_F)=\Crit(F)$. Thus the spectral invariant $c((\overline{H'}\oplus H')\#G\# F,a \otimes b)$ can be expressed as follows:

$$c((\overline{H'}\oplus H')\#G\# F,a \otimes b)$$
$$=F(x,y)+\mathcal{A}_{(\overline{H'}\oplus H')\#G}([\phi_{(\overline{H'}\oplus H')\#G} ^t((x,y)), w_{x,y}] )+(\omega \oplus \omega)(A_1)$$
where
\begin{itemize}
\item  $(x,y)$ is a certain critical point of $F$ which is located in $B\times B$.
\item $w_{x,y}$ denotes an arbitrary chosen capping of the orbit $\phi_{(\overline{H'}\oplus H')\#G} ^t((x,y))$. We fix this capping in the sequel.
\item $A_1 $ denotes the sphere which plays the role of correcting the capping $w_{x,y}$
\end{itemize}
\end{proof}

$\bullet$ Step 3: The aim of this step is to prove the following:

\begin{claim}
$$\mathcal{A}_{(\overline{H'}\oplus H')\#G}([\phi_{(\overline{H'}\oplus H')\#G} ^t((x,y)),w_{x,y}]) = (\omega \oplus \omega)(A_2)$$
for some $A_2 \in \pi_2(M \times M)$.
\end{claim}

\begin{proof}
By Proposition \ref{homotopy invariance} (2), we obtain
$$\mathcal{A}_{(\overline{H'}\oplus H')\#G}([\phi_{(\overline{H'}\oplus H')\#G} ^t((x,y)),w_{x,y}])=\mathcal{A}_{(\overline{H'}\oplus H')\#G}([\phi_{(\overline{H'}\oplus H')\#G} ^t((x_0,x_0)), w_{x_0,x_0}])$$
where $w_{x_0,x_0}$ is the capping of the orbit $\phi_{(\overline{H'}\oplus H')\#G} ^t((x_0,x_0))$ corresponding to the capping $w_{x,y}$ in the sense of Proposition \ref{homotopy invariance} (2). As $Q$ is a Hamitonian which generates a time-1 map that switches the coordinate i.e. $(p,q)\mapsto (q,p)$ in $B\times B$ and satisfies $\forall t, \phi_Q ^t ((x_0,x_0)) = (x_0,x_0),$ we have

$$ \mathcal{A}_{(\overline{H'}\oplus H')\#G}( [ \phi_{(\overline{H'}\oplus H')\#G} ^t((x_0,x_0)), w_{x_0,x_0}] )$$
$$= \int Q(\phi_Q ^t (x_0,x_0))dt
 + \int (0\oplus H')(t,x_0,\phi_{H'} ^t(x_0))dt-\omega( \overline{\phi_{H'} ^t(x_0)}) +$$
$$\int  \overline{Q}(\phi_{ \overline{Q}} ^t (x_0,x_0)dt
+\int (0\oplus \overline{H'})(t,x_0,\phi_{\overline{H'}} ^t(x_0))dt-\omega(\overline{\phi_{\overline{H'}} ^t(x_0)})+(\omega \oplus \omega)(A_2)$$

where 
\begin{itemize}
\item $\overline{\phi_{H'} ^t(x_0)}$ denotes the capped orbit of $\phi_{H'} ^t(x_0)$ whose capping is chosen arbitrarily.
\item $ \overline{\phi_{\overline{H'}} ^t(x_0)}$ denotes the capped orbit of $\phi_{\overline{H'}} ^t(x_0)$ whose capping is the same as the the capping of $\phi_{H'} ^t(x_0)$ chosen above.
\item $A_2$ denotes the sphere to which corrects the capping of the RHS so that it will meet the capping on the LHS.
\end{itemize}
Thus, by employing Proposition \ref{opposite direction} (3) for $\int H' _t(\phi_{H'} ^t(x_0))dt$ and $ \int \overline{H'} _t(\phi_{\overline{H'}} ^t(x_0))dt,$ we obtain,
$$\displaystyle \mathcal{A}_{(\overline{H'}\oplus H')\#G}( [ \phi_{(\overline{H'}\oplus H')\#G} ^t((x_0,x_0)), w_{x_0,x_0}] )=(\omega \oplus \omega)(A_2).$$
\end{proof}

$\bullet$ Step 4: The aim of this step is to complete the proof.

By Step 2 and 3, we have
$$c((\overline{H'}\oplus H')\#G\#F,a \otimes b)=F(x,y)+(\omega \oplus \omega)(A_2)+(\omega \oplus \omega)(A_1)$$
$$=F(x,y)+l\cdot \lambda_0$$
for some integer $l\in \mathbb{Z}$ such that $(\omega \oplus \omega)(A_1+A_2)=l\cdot \lambda_0$ and
$$c(\overline{H'}\oplus H',a \otimes b)=\gamma_{a,b}(H')=\gamma_{a,b}(H\circ \psi)=\gamma_{a,b}(H)$$
where the last equality uses Proposition \ref{prop spec inv} (7).

By Step 1, we conclude that
$$|\gamma_{a,b}(H)-l \cdot \lambda_0|\leq \frac{5}{2}\varepsilon.
$$
Hence we complete the proof.

\end{proof}

\subsection{Proof of Theorem \ref{C0-conti of gamma:CPn}}\label{monotone proof}
The aim of this section is to prove Theorem \ref{C0-conti of gamma:CPn}. We prove the following a priori more general result.

\begin{theo}\label{monotone bound}
Let $(M^{2n},\omega)$ be a monotone symplectic manifold with a minimal Chern number $N_M>n$. Assume that there exist $\psi \in \pi_1(\Ham(M,\omega))$ and a section class $\sigma$ of the Hamiltonian fibration $M_{\psi} \to S^2$, such that its Seidel element $\mathcal{S}_{\psi,\sigma} \in QH_\ast(M;\K)$ satisfies the following:
\begin{itemize}
\item $(\mathcal{S}_{\psi,\sigma})^{\ast k}=a_1\cdot [pt]$ for some $a_1 \in \K \backslash \{0\}$ and $k\in \mathbb{N}$ where $[pt]$ denotes the point class in $H_0(M;\K)$. 
\item  $(\mathcal{S}_{\psi,\sigma})^{\ast k'}=a_2 \cdot [M]\cdot s^{-l'}$ for some $a_2 \in \K \backslash \{0\}$ and $k',l'\in \mathbb{N}$ where $[M]$ denotes the fundamental class and $s$ denotes the generator of the Novikov ring of $(M,\omega)$.
\end{itemize}
Then the spectral norm satisfies the following.
\begin{enumerate}
\item For any $\phi \in \Ham(M,\omega),$
$$\gamma(\phi)\leq \frac{n}{N_M}\cdot \lambda_0.$$
\item The spectral norm is $C^0$-continuous i.e.
$$\gamma:(\Ham(M,\omega),d_{C^0}) \to \mathbb{R}$$
is continuous. Moreover, $\gamma$ extends continuously to $\overline{\Ham}(M,\omega).$
\end{enumerate}
\end{theo}

\begin{remark}
\begin{enumerate}
\item Theorem \ref{monotone bound} (1) is essentially contained in Proposition 15 in \cite{[KS18]} where Kislev-Shelukhin considers Lagrangian spectral invariants instead of Hamiltonian ones.
\item So far, $(\mathbb{C}P^n,\omega_{FS})$ seems to be the only example that satisfies the assumptions in Theorem \ref{monotone bound}.
\end{enumerate}
\end{remark}

\begin{proof}(of Theorem \ref{monotone bound})

Let $\phi \in \Ham(M,\omega)$ and take any Hamiltonian $H$ such that $\phi_H=\phi.$ Let $\psi \in \pi_1(\Ham(M,\omega))$ and $\sigma $ be as in the statement. Denote 
$$a:=\mathcal{S}_{\psi,\sigma} \in QH_\ast(M;\K),\ a^{\ast k}:=\underbrace{a\ast a \ast \cdots \ast a}_{k-times}.$$
 
By looking at the degree, we have
\begin{itemize}
\item  $\deg (a^{\ast k})=\deg([pt])=0$,
\item $\deg (a^{\ast k'})=\deg([M]\cdot s^{-l'})=2n-2Nl'$,
\item For any $m\in \mathbb{N},$ $\deg( a ^{\ast m})=m\cdot \deg (a)-(m-1)\cdot 2n.$
\end{itemize}
These equations will give us the following:
\begin{equation}\label{degree}
\frac{k'}{k}=\frac{Nl'}{n}
\end{equation}
and our assumption $N>n$ implies $k'>k.$ As $N_M>n$ and $\K$ is a field, the formula in \cite{[EP03]} Section 2.7 gives us
$$c(\ovl{H}, [M]) = - c(H,[pt]),$$
and by Proposition \ref{prop spec inv}, we get the following. 
\begin{itemize}
\item $\gamma(H)=c(H,[M])-c(H,[pt])=c(H,[M])-c(H,a^{\ast k}),$\\

\item $\gamma(\psi^\ast H)=c(H,\mathcal{S}_{\psi,\sigma}\ast [M])-c(H,\mathcal{S}_{\psi,\sigma}\ast a^{\ast k}))=c(H,a)-c(H,a^{\ast (k+1)}).$\\

\item $\gamma((\psi^2)^\ast H)=c(H,a^{\ast 2})-c(H,a^{\ast (k+2)}).$\\

$\cdots$\\

\item $\gamma((\psi^{k'-k})^\ast H)=c(H,a^{\ast (k'-k)})-c(H,a^{\ast k'})$
$$=c(H,a^{\ast (k'-k)})-c(H,[M])+l'\lambda_0.$$
\item $\gamma((\psi^{k'-k+1})^\ast H)=c(H,a^{\ast (k'-k+1)})-c(H,a)+l'\lambda_0.$\\

$\cdots$\\

\item $\gamma((\psi^{k'-1})^\ast H)=c(H,a^{\ast (k'-1)})-c(H,a^{\ast (k-1)})+l'\lambda_0.$

\end{itemize}

We used that for $j\in \Z$,
$$c(H,a^{\ast (j+k')})=c(H,a^{\ast j})- l'\lambda_0.$$

Adding up these $k'$-equations will give us the following.
$$\sum_{0\leq j\leq k'-1}\gamma((\psi^j)^\ast H)=kl'\cdot \lambda_0.$$
As $\gamma (\phi) \leq \gamma((\psi^j)^\ast H)$ for all $0\leq j\leq k'-1$,
$$k'\cdot \gamma(\phi)\leq kl'\cdot \lambda_0.$$
By equation \ref{degree}, we conclude
$$\gamma(\phi) \leq \frac{kl'}{k'}\cdot \lambda_0= \frac{n}{N}\cdot \lambda_0.$$ 
The continuity of $\gamma$ is a direct consequence of Corollary \ref{bound implies conti}.
\end{proof}

Theorem \ref{C0-conti of gamma:CPn} is a direct consequence of Theorem \ref{monotone bound}.

\begin{proof}(of Theorem \ref{C0-conti of gamma:CPn})

We explain briefly that $\mathbb{C}P^n$ meets the assumptions in Theorem \ref{monotone bound}. Consider a loop of Hamiltonian diffeomorphism of $\mathbb{C}P^n$ defined by $$\psi^t ([z_0:z_1:\cdots:z_{n-1}:z_n]):= [z_0:e^{2\pi i t}z_1:e^{2\pi i t} z_2:\cdots:e^{2\pi i t}z_{n-1}:e^{2\pi i t}z_n].$$ It is known that there exists a section class $\sigma$ such that $\mathcal{S}_{\psi,\sigma}=[\C P^{n-1}]$ where $[\C P^{n-1}]$ denotes the generator of $H_{2n-2}(\C P^{n};\K).$ See Example 9.6.1 and Proposition 9.6.4 in \cite{[MS04]}. This shows that $\mathbb{C}P^n$ satisfies the assumptions in Theorem \ref{monotone bound}. 
\end{proof}

\section{Proofs of applications}

\subsection{The displaced disks problem}

We prove Theorem \ref{ddp}. We use the following energy-capacity inequality proven by Usher in \cite{[Ush10]}.

\begin{prop}\label{usher}$($\cite{[Ush10]}$)$

Let $B:=B(r)$ be an open ball in $(\R^{2n},\omega_{std})$. If $B(r)$ is symplectically embedded to $(M,\omega)$ 
$$f: B(r) \hookrightarrow (M,\omega)$$
and $\phi(f(B))\cap f(B)= \emptyset$ for $\phi \in \Ham(M,\omega)$, then
$$\pi r^2 \leq \gamma(\phi).$$
\end{prop}

Notice that for $(M,\omega)$ for which the spectral norm is $C^0$-continuous, Proposition \ref{usher} holds for Hamiltonian homeomorphisms as well.

\begin{proof}(of Theorem \ref{ddp})

By Theorem \ref{neg mono}, we can apply Proposition \ref{usher} for Hamiltonian homeomorphisms. Let $r>0$ and take $\delta>0$ so that if $\phi \in \ovl{\Ham}(M,\omega), \gamma(\phi) \geq \pi r^2$, then $d_{C^0}(\id, \phi) >\delta$. Now, we will prove that if $\phi \in \ovl{\Ham}(M,\omega)$ displaces an embedded ball of radius $r$, then $d_{C^0}(\id, \phi) >\delta$. By Proposition \ref{usher}, we have
$\gamma(\phi) \geq \pi r^2$ and from our choice of $\delta$, this implies $d_{C^0}(\id, \phi) >\delta$.
\end{proof}

\subsection{The $C^0$-Arnold conjecture}\label{section-arnold}
We start by looking at properties of $\sigma_{a,a \ast b}$ defined earlier in Section \ref{App2}.

\begin{prop}\label{triangle ineq}
Let $(M^{2n},\omega)$ be a symplectic manifold and $a,b\in H_\ast(M;\K)\backslash \{0\}$. For Hamiltonians $H,G$, we have the following triangle inequality:
$$|\sigma_{a,a \ast b}(H)-\sigma_{a,a \ast b}(G)|\leq \gamma(\overline{H}\#G).$$
\end{prop}

\begin{proof}
$$\sigma_{a,a \ast b}(H)-\sigma_{a,a \ast b}(G)=c(H,a)-c(H,a \ast b)-(c(G,a)-c(G, a\ast b))$$
$$\leq c(\overline{G}\#H,[M])+c(\overline{H}\#G,[M])=\gamma(\overline{H}\#G).$$ By changing the role of $H$ and $G$, we get $\sigma_{a , a \ast b}(G)-\sigma_{a,a \ast b}(H)\leq \gamma(\overline{H}\#G)$ too. This completes the proof.
\end{proof}

Proposition \ref{triangle ineq} allows us to define the following: Let $(M^{2n},\omega)$ be a negative monotone symplectic manifold and $a,b\in H_\ast(M;\K)$.
$$\sigma_{a,a \ast b}: \Ham(M,\omega) \to \R$$
$$\sigma_{a,a \ast b}(\phi):=\sigma_{a,a \ast b}(H)$$
for any $H$ such that $\phi_H=\phi.$ Note that the well-definedness is due to Theorem \ref{neg mono}. Similarly, we define the following for $\mathbb{C}P^n$: Let $h:=[\C P^{n-1}]$ and $l_1,l_2 \in \N, l_1 < l_2$.
$$\sigma_{h^{l_1}, h^{l_2}}: \Ham(\C P^n,\omega) \to \R$$
$$\sigma_{h^{l_1}, h^{l_2}}(\phi):= \inf_{\phi_H=\phi} \sigma_{h^{l_1}, h^{l_2}} (H) .$$

\begin{corol}
Let $(M^{2n},\omega)$ be either a negative monotone symplectic manifold or $(\mathbb{C}P^n,\omega_{FS})$. For $a,b \in H_\ast (M;\K)$, we have the following triangle inequality: For $\phi,\psi \in \Ham(M,\omega)$,
$$ |\sigma_{a,a \ast b}(\phi)-\sigma_{a, a\ast b}(\psi)|\leq \gamma(\phi^{-1}\psi).$$
\end{corol}

\begin{proof}
We only explain the case of $(\mathbb{C}P^n,\omega_{FS})$ since the other is simpler. By Proposition \ref{triangle ineq},
$$\sigma_{h^{l_1}, h^{l_2}} (H\#G)\leq \sigma_{h^{l_1}, h^{l_2}} (H)+\gamma(G).$$
Take an infimum on both sides as in the definition.
$$\sigma_{h^{l_1}, h^{l_2}} (\phi\psi)\leq \inf_{\phi_H=\phi,\phi_G=\psi} \sigma_{h^{l_1}, h^{l_2}} (H\#G)\leq \sigma_{h^{l_1}, h^{l_2}} (\phi)+\gamma(\psi).$$
Since $\sigma_{h^{l_1}, h^{l_2}}$ are finite,
$$\sigma_{h^{l_1}, h^{l_2}} (\phi\psi)- \sigma_{h^{l_1}, h^{l_2}} (\phi)\leq \gamma(\psi).$$
This implies the triangle inequality
$$ |\sigma_{h^{l_1}, h^{l_2}} (\phi)- \sigma_{h^{l_1}, h^{l_2}} (\psi)|\leq \gamma(\phi^{-1}\psi)$$
where $\phi,\psi \in \Ham(\C P^n,\omega_{FS})$.
\end{proof}

This corollary and the $C^0$-continuity of $\gamma$ implies the $C^0$-continuity of $\sigma_{a,a\ast b}$. This allows us to define $\sigma_{a,a \ast b}$ for Hamiltonian homeomorphisms i.e. for a Hamiltonian homeomorphism $\phi$, define $\sigma_{a, a \ast b}(\phi):=\lim_{n\to \infty} \sigma_{a, a \ast b}(\phi_n)$ where $\phi_n \in \Ham(M,\omega),\ \phi_n \xrightarrow{C^0} \phi$.

We are now ready to prove Theorem \ref{theo-arnold}. 

\begin{proof}(of Theorem \ref{theo-arnold})

Since the negative monotone case is simpler than the case of $(\mathbb{C}P^n,\omega_{FS})$, we only prove the latter. We assume that for $\phi \in \ovl{\Ham}(\C P^n, \omega_{FS})$ and $l_1 < l_2$, we have
$$\sigma_{h^{l_1}, h^{l_2}} (\phi) = 0.$$
It is enough to prove that an arbitrary open neighborhood $U$ of $\Fix(\phi)$ is homologically non-trivial. Let $f:M\to \mathbb{R}$ be a  sufficiently $C^2$-small smooth function such that $f< 0$ on $M\backslash \ovl{U}$, $f|_{U}=0$ and $c_{LS}(f,\cdot)=c(f,\cdot).$ (See Proposition \ref{prop spec inv} (5) for the definition of $c_{LS}$.)

First of all, take a sequence $\phi_j \in \Ham(M,\omega),\ j\in \mathbb{N}$ such that 
$$d_{C^0}(\phi,\phi_j)\leq 1/j.$$ 
The $C^0$-continuity of $\gamma$ allows us to take a subsequence $\{j_k\}_{k\in \mathbb{N}}$ so that for each $k$, 
$$\gamma(\phi^{-1}\phi_{j_k})<1/k.$$
 Next, for each $k$, take a Hamiltonian $H_k$ which generates $\phi_{j_k}$ and
 $$\sigma_{h^{l_1}, h^{l_2}}  (H_k)\leq \sigma_{h^{l_1}, h^{l_2}} (\phi_{j_k})+1/k.$$

We borrow the following claim proved in  \cite{[BHS18b]}.
\begin{claim}$($Claim 5.3 in \cite{[BHS18b]}$)$
Assume $\phi_{H_k} \xrightarrow{C^0} \phi$. For any $a \in H_\ast (M ; \K) \backslash \{0\}$, there exists $0< \eps _0<1$ and an integer $k_0$ such that for any $k \geq k_0$, we have $c(H_k \#\eps_0 f , a) =c(H_k , a).$
\end{claim}

 From this Claim, there exist $\eps_0 >0$ and $k_0\in \mathbb{N}$ such that if $k\geq k_0$, then 
 $$c(H_k\# \eps_0 f,a )=c(H_k,a)$$ for all $a\in H_\ast(\C P^n;\K).$
For $k\geq k_0$,
$$c(H_k, h^{l_2})=c(H_k \#\eps_0 f,h^{l_2} )\leq c(H_k,h^{l_1})+c(\eps_0 f, h^{l_2-l_1})$$
and thus,
$$-\sigma_{h^{l_1}, h^{l_2}} (H_k) \leq c(\eps_0 f, h^{l_2-l_1}) \leq c(f, h^{l_2-l_1}).$$
By our choices of $\phi_{j_k}$ and $H_k$, we have the following.
$$ \sigma_{h^{l_1}, h^{l_2}} (H_k)\leq \sigma_{h^{l_1}, h^{l_2}} (\phi_{j_k})+1/k \leq \sigma_{h^{l_1}, h^{l_2}} (\phi)+\gamma(\phi^{-1}\phi_{j_k})+1/k $$
$$\leq \sigma_{h^{l_1}, h^{l_2}} (\phi)+2/k=2/k.$$
Thus,
$$-2/k\leq  -\sigma_{h^{l_1}, h^{l_2}}(H_k) \leq c(f,h^{l_2-l_1}).$$
By taking a limit $k\to +\infty$, we obtain
$$0\leq c(f,h^{l_2-l_1}).$$
Thus,
$$0\leq c(f,h^{l_2-l_1})\leq c(f,[M])\leq 0.$$
The last inequality follows from $f\leq 0.$ Since $f$ was taken to satisfy $c_{LS}(f,\cdot)=c(f,\cdot),$ we have
$$c_{LS}(f,h^{l_2-l_1})=c_{LS}(f,[M])(=0).$$
This implies that $\ovl{U}$ is homologically non-trivial.
\end{proof}

\appendix


\begin{thebibliography}{2}
\bibitem[BHS18a]{[BHS18a]} Lev Buhovsky, Vincent Humili\`ere, Sobhan Seyfaddini \emph{A $C^0$-counterexample to the Arnold conjecture}, Invent. Math. 213 (2018), no. 2, 759-809


\bibitem[BHS18b]{[BHS18b]} Lev Buhovsky, Vincent Humili\`ere, Sobhan Seyfaddini \emph{The action spectrum and $C^0$-symplectic topology}, ArXiv:1808.09790v1


\bibitem[EP03]{[EP03]} Michael Entov, Leonid Polterovich,  \emph{Calabi quasimorphism and quantum homology}, Int. Math. Res. Not. 2003, no. 30, 1635--1676. 


\bibitem[EP09]{[EP09]} Michael Entov, Leonid Polterovich,  \emph{Rigid subsets of symplectic manifolds}, Compos. Math. 145 (2009), no. 3, 773--826.



\bibitem[Fl89]{[Fl89]} Andreas Floer, \emph{Symplectic fixed points and holomorphic spheres}, Comm. Math. Phys., 120(4):575--611, 1989.


\bibitem[For85]{[For85]} Barry Fortune, \emph{A symplectic fixed point theorem for $\C P^n$}, Invent. Math., 81(1):29--46, 1985.


\bibitem[ForW85]{[ForW85]} Barry Fortune and A. Weinstein, \emph{A symplectic fixed point theorem for complex projective spaces}, Bull. Amer. Math. Soc. (N.S.), 12(1):128--130, 1985.


\bibitem[FO99]{[FO99]} Kenji Fukaya, Kaoru Ono,  \emph{Arnold conjecture and Gromov-Witten invariant}, Topology 38 (1999), no. 5, 933--1048.



\bibitem[FOOO09]{[FOOO09]} Kenji Fukaya, Yong-Geun Oh, Hiroshi Ohta, Kaoru Ono, \emph{Lagrangian intersection Floer theory: anomaly and obstruction. Part I. II.}, AMS/IP Studies in Advanced Mathematics, 46.1. American Mathematical Society, Providence, RI; International Press, Somerville, MA, 2009.


\bibitem[How12]{[How12]} Wyatt Howard, \emph{Action Selectors and the Fixed Point Set of a Hamiltonian Diffeomorphism}, arXiv:1211.0580v2


\bibitem[KS18]{[KS18]} Asaf Kislev, Egor Shelukhin, \emph{Bounds on spectral norms and barcodes}, arXiv:1810.09865



\bibitem[LO94]{[LO94]} H\^ong V\^an L\^e, Kaoru Ono, \emph{Cup-length estimates for symplectic fixed points}, Contact and symplectic geometry (Cambridge, 1994), 268--295.

\bibitem[LSV18]{[LSV18]} Fr\'ed\'eric Le Roux, Sobhan Seyfaddini, Claude Viterbo, \emph{Barcodes and area-preserving homeomorphisms}, arXiv:1810.03139


\bibitem[LT98]{[LT98]} Gang Liu, Gang Tian, \emph{Floer homology and Arnold conjecture}, 
J. Differential Geom. 49 (1998), no. 1, 1--74. 




\bibitem[LZ18]{[LZ18]} R\'emi Leclercq, Frol Zapolsky, \emph{Spectral invariants for monotone Lagrangians}, J. Topol. Anal. 10 (2018), no. 3, 627--700.



\bibitem[Ma00]{[Ma00]} Shigenori Matsumoto, \emph{Arnold conjecture for surface homeomorphisms}, In Proceedings of the French-Japanese Conference "Hyperspace Topologies and Applications" (La Bussi\`ere, 1997), volume 104, pages 191--214, 2000.


\bibitem[McD10]{[McD10]} Dusa McDuff, \emph{Monodromy in Hamiltonian Floer theory}, Comment. Math. Helv. 85 (2010), 95--133


\bibitem[MS98]{[MS98]} Dusa McDuff, Dietmar Salamon, \emph{Introduction to symplectic topology}, Second edition.  Oxford Mathematical Monographs. The Clarendon Press, Oxford University Press, New York, 1998


\bibitem[MS04]{[MS04]} Dusa McDuff, Dietmar Salamon, \emph{J
 -holomorphic Curves and Symplectic Topology: Second Edition}, American Mathematical Society Colloquium Publications, 52. American Mathematical Society, Providence, RI, 2004 
 
\bibitem[Oh05]{[Oh05]} Yong-Geun Oh, \emph{Construction of spectral invariants of Hamiltonian paths on closed symplectic manifolds}, The breadth of symplectic and Poisson geometry, 525--570, Progr. Math., 232, Birkh\"auser Boston, Boston, MA, 2005.


\bibitem[Ost06]{[Ost06]} Yaron Ostrover, \emph{Calabi quasi-morphisms for some non-monotone symplectic manifolds}, Algebr. Geom. Topol. 6 (2006), 405--434.


\bibitem[PS16]{[PS16]} Leonid Polterovich, Egor Shelukhin, \emph{ Autonomous Hamiltonian flows, Hofer's geometry and persistence modules}, Selecta Math. (N.S.) 22 (2016), no. 1, 227--296. 


\bibitem[PSS17]{[PSS17]} Leonid Polterovich, Egor Shelukhin, Vukasin Stojisavljevi\'c \emph{Persistence modules with operators in Morse and Floer theory}, Moscow Mathematical Journal, 17 (2017), no. 4, 757--786.


\bibitem[Sch00]{[Sch00]} Matthias Schwarz, \emph{On the action spectrum for closed symplectically aspherical manifolds}, Pacific J. Math. 193 (2000), no. 2, 419--461.


\bibitem[Sei97]{[Sei97]} Paul Seidel, \emph{$\pi_1$ of symplectic automorphism groups and invertibles in quantum homology rings}, Geom. Funct. Anal. 7 (1997), no. 6, 1046--1095. 


\bibitem[Sey13-1]{[Sey13-1]}  Sobhan Seyfaddini, \emph{$C^0$-limits of Hamiltonian paths and the Oh-Schwarz spectral invariants}, Int. Math. Res. Not. IMRN 2013, no. 21, 4920--4960.


\bibitem[Sey13-2]{[Sey13-2]}  Sobhan Seyfaddini, \emph{The displaced disks problem via symplectic topology},  C. R. Math. Acad. Sci. Paris 351 (2013), no. 21-22, 841--843. 


\bibitem[Sh18]{[Sh18]} Egor Shelukhin, \emph{Viterbo conjecture for Zoll symmetric spaces}, ArXiv:1811.05552


\bibitem[Ush08]{[Ush08]} Michael Usher, \emph{Spectral numbers in Floer theories}, Compos. Math. 144 (2008), no. 6, 1581--1592.


\bibitem[Ush10]{[Ush10]} Michael Usher, \emph{The sharp energy-capacity inequality}, Commun. Contemp. Math. 12 (2010), no. 3, 457--473.



\bibitem[UZ16]{[UZ16]} Michael Usher, Jun Zhang, \emph{Persistent homology and Floer-Novikov theory}, Geom. Topol. 20 (2016), no. 6, 3333--3430. 



\bibitem[Vit92]{[Vit92]} Claude Viterbo, \emph{Symplectic topology as the geometry of generating functions}, Math. Ann. 292 (1992), no. 4, 685--710. 


\bibitem[Zap15]{[Zap15]} Frol Zapolsky, \emph{The Lagrangian Floer-quantum-PSS package and canonical orientations in Floer theory},  arXiv:1507.02253v2.









\end{thebibliography}
\end{document}